\documentclass[12pt,UTF-8,reqno]{amsart}
\usepackage{enumerate, bbm}
\usepackage{amssymb,url,color, booktabs}
\usepackage{mathrsfs}

\usepackage{color}
\usepackage[colorlinks=true]{hyperref}
\hypersetup{
    linkcolor=blue,          
    citecolor=red,        
    filecolor=blue,      
    urlcolor=cyan
}

\usepackage{color}
\definecolor{MyDarkBlue}{cmyk}{0.8,0.3,0.8,0.4}
\definecolor{yellow}{rgb}{0.99,0.99,0.70}
\definecolor{white}{rgb}{1.0,1.0,1.0}
\definecolor{black}{rgb}{0.00,0.00,0.00}

\numberwithin{equation}{section}

\newcommand{\be}{\begin{eqnarray}}
\newcommand{\ee}{\end{eqnarray}}
\newcommand{\ce}{\begin{eqnarray*}}
\newcommand{\de}{\end{eqnarray*}}
\newtheorem{theorem}{Theorem}[section]
\newtheorem{lemma}{Lemma}[section]
\newtheorem{remark}{Remark}[section]
\newtheorem{definition}{Definition}[section]
\newtheorem{proposition}{Proposition}[section]

\newtheorem{corollary}{Corollary}[section]

\newtheorem{Rem}{Remark}[section]

\def\[{{\Big[}}
\def\]{{\Big]}}
\def\<{{\langle}}
\def\>{{\rangle}}
\def\({{\big(}}
\def\){{\big)}}

\def\bb2{{\boldsymbol{2}}}

\def\={&\!\!=\!\!&}
\def\RR{\mathbb{R}}

\def\b1{{\mathbbm 1}}

\def\geq{\geqslant}
\def\leq{\leqslant}
\def\ge{\geqslant}
\def\le{\leqslant}

\def\[{{\Big[}}
\def\]{{\Big]}}
\def\<{{\langle}}
\def\>{{\rangle}}

\def\={&\!\!=\!\!&}
\def\bt{\begin{theorem}}
\def\et{\end{theorem}}
\def\bl{\begin{lemma}}
\def\el{\end{lemma}}
\def\br{\begin{remark}}
\def\er{\end{remark}}
\def\bd{\begin{definition}}
\def\ed{\end{definition}}
\def\bc{\begin{corollary}}
\def\ec{\end{corollary}}

\def\geq{\geqslant}
\def\leq{\leqslant}
\def\ge{\geqslant}
\def\le{\leqslant}

\def\<{\langle} \def\>{\rangle}

\def\X{{\widetilde X}}

\allowdisplaybreaks

\begin{document}

\title[LDP for  Slow-Fast Mean-Field Diffusions]
{Large Deviations for Slow-Fast Mean-Field Diffusions$^\dagger$}

\thanks{$\dagger$
This work is supported by National Key R\&D program of China (No. 2023YFA1010101). The research of W. Hong is also supported by  NSFC (No.~12401177) and Basic Research Program of Jiangsu (No.~BK20241048).  The research of W. Liu is also supported by NSFC (No.~12571155). 
}

\maketitle
\centerline{\scshape Wei Hong$^a$,   Wei Liu$^{a,}$\footnote{~Corresponding author: weiliu@jsnu.edu.cn},  Shiyuan Yang$^{a}$ }
\medskip
 {\footnotesize
\centerline{ $a.$   School of Mathematics and Statistics, Jiangsu Normal University, Xuzhou, 221116, China}

\begin{abstract}
The aim of this paper is to investigate the large deviations for a class of slow-fast mean-field diffusions, where the slow process depends on not only the fast process but also their distributions.  Due to the perturbations of fast
process and its time marginal law, one cannot prove the large deviations based on verifying the powerful weak convergence criterion directly. To overcome this problem, we employ the functional occupation measure, which combined with the notion of the viable pair and the controls of feedback form  to characterize the limits of  controlled sequences and justify the upper and lower bounds of Laplace principle. As a consequence, the explicit representation formula of the rate function for large deviations is also presented.


\end{abstract}
\maketitle \rm

\tableofcontents

\section{Introduction}
In this work,  we focus on the asymptotical behavior of the following mean-field equation with fast oscillations
\begin{equation}\label{n1}
dX_t^{\delta}=b(X_t^{\delta},\mathscr{L}_{X_t^{\delta}},Y_t^{\delta},\mathscr{L}_{Y_t^{\delta}})dt+\sqrt{\delta}\sigma(X_t^{\delta},\mathscr{L}_{X_t^{\delta}},Y_t^{\delta},\mathscr{L}_{Y_t^{\delta}})dW_t^1,~X_0^{\delta}=x,
\end{equation}
where the fast oscillating process $Y_t^{\delta}$ satisfies the equation
\begin{equation}\label{n0}
dY_t^{\delta}=\frac{1}{\varepsilon}f(Y_t^{\delta})dt+\frac{1}{\sqrt{\varepsilon}}g(Y_t^{\delta})dW_t^2,~Y_0^{\delta}=y.
\end{equation}
Here, $\mathscr{L}_{X_t^{\delta}}$ (resp.~$\mathscr{L}_{Y_t^{\delta}}$) denotes the law of $X_t^{\delta}$ (resp.~$Y_t^{\delta}$),  the scale $\delta$ describes the intensity of the noise and $\varepsilon:=\varepsilon(\delta) $ describes the ratio of the time scale between the (slow) component $X^{\delta}$ and the (fast) component $Y^{\delta}$. We define an $\mathbb{R}^{d_1+d_2}$-valued standard Brownian motion $W$ on a complete filtration probability space $\left(\Omega,\mathcal{F},(\mathcal{F}_{t})_{t\in[0,T]},\mathbb{P}\right)$. Then, we can choose the projection operators $P_1:\mathbb{R}^{d_1+d_2}\to \mathbb{R}^{d_1}, P_2:\mathbb{R}^{d_1+d_2}\to \mathbb{R}^{d_2}$ such that
\begin{equation}\nonumber
W_t^1:=P_1W_t,~ W_t^2:=P_2W_t
\end{equation}
are independent $d_1$ and $d_2$ dimensional standard Brownian motions, respectively.

Such systems (\ref{n1})-(\ref{n0}) are called the multi-scale (or slow-fast) dynamical systems, which have extensive applications in various fields such as  climate dynamics,  molecular dynamics and mathematical finance (see e.g.~\cite{a,l,q,mm}). Moreover, studying these systems with small random noise perturbations not only has potential applications, but also demonstrates theoretical challenges due to the interactions between different scales. For instance, exploring rare transition events among equilibrium states in multi-scale dynamical systems (cf.~\cite{ooo, ggg, zzz}) and examining the small-time asymptotics of multi-scale financial models \cite{s, t} are of particular interest in mathematics and finance.

Our goal of the present work is to study the large deviation principle (LDP for short), as $\delta\to 0$, for the mean-field diffusions (\ref{n1}).
The LDP is a classical topic in probability theory, primarily describing the asymptotic behavior of the remote tails of a family of probability distributions. It has extensive applications in various fields, including information theory, thermodynamics, statistics and engineering. When dealing with stochastic processes, a central concept in studying LDP is to identify a deterministic path around which the diffusion process is highly concentrated. This kind of asymptotic behaviours is called the small perturbation type LDP (also known as Freidlin-Wentzell's LDP), which was first introduced by Freidlin and Wentzell  for stochastic differential equations (SDEs for short) in their pioneering work \cite{u}. This framework has been extensively studied over the past several decades. For more detailed expositions on the background and applications for the theory of large deviations, we refer to the classical monographs  \cite{ppp, ttt, ww}.

On the other hand, the mean-field stochastic equations  have attracted widespread attention due to their effective applications to describe stochastic systems whose evolution is influenced by both the microscopic locations and the macroscopic laws of particles.
  There have numerous results in literature on mean-field stochastic equations in recent years, for example, one could refer to \cite{ee, rr, xx} for the well-posedness results  and  \cite{p, kk, qq}  for several asymptotic properties.
 The investigation of mean-field stochastic equations and interacting particle systems can be traced back to Kac's program in kinetic theory \cite{hh} and McKean's seminal work \cite{oo} on nonlinear parabolic equations. For example, we analyze the dynamics of the $N$-particle system governed by the following SDEs
\begin{equation*}\label{zz1.1}
dX_t^{N,i}=b(X_t^{N,i},\mu_t^N)dt+\sigma(X_t^{N,i},\mu_t^N)dW_t^i,~~\mu_t^N=N^{-1}\sum_{j=1}^N\delta_{X_t^{N,j}},
\end{equation*}
where $i = 1, \cdots, N,$ the mean field interactions are expressed through the dependence of coefficients on the empirical laws $\mu_t^N$ of the system. Under some appropriate assumptions on coefficients and the exchangeability assumption on initial conditions, as $N \to \infty$, the empirical laws $\mu_t^N$  weakly converge to the law of the solution to the following mean-field stochastic equation
\begin{equation*}
dX_t=b(X_t,\mathscr{L}_{X_t})dt+\sigma(X_t,\mathscr{L}_{X_t})dW_t.
\end{equation*}
Thus, the coefficients of the limiting equation will inherently depend not only on the solution pointwisely but also on its time marginal law. This kind of limiting behavior is commonly referred as the propagation of chaos in the study of stochastic dynamics of particle systems.

Due to the above reasons,  many scholars are interested in the multi-scale interacting particle systems. For example, one often considers the following form of system
\begin{equation}\label{zz1.2}
	\left.\left\{\begin{aligned}
		&dX_t^{\varepsilon,N,i}=b(X_t^{\varepsilon,N,i},\mu_t^{\varepsilon,N},X_t^{\varepsilon,N,i}/\varepsilon,\nu_t^{\varepsilon,N})dt +\sigma(X_t^{\varepsilon,N,i},\mu_t^{\varepsilon,N},X_t^{\varepsilon,N,i}/\varepsilon,\nu_t^{\varepsilon,N})dW_t^i,\\
		&\mu_t^{\varepsilon,N}=N^{-1}\sum_{j=1}^N\delta_{X_t^{\varepsilon,N,j}},~\nu_t^{\varepsilon,N}=N^{-1}\sum_{j=1}^N\delta_{X_t^{\varepsilon,N,j}/\varepsilon},
	\end{aligned}\right.\right.
\end{equation}
where $i = 1,\cdots, N$, $\varepsilon$ is a small parameter indicating the ratio of time scales. The variables $X_t^{\varepsilon,N,i}$ and $Y_t^{\varepsilon,N,i} := X_t^{\varepsilon,N,i}/\varepsilon$ represent the slow and fast components in system (\ref{zz1.2}), respectively. Investigating the combined mean field and homogenization limits (i.e. $N \to \infty$ and $\varepsilon \to 0$) for the multi-scale interacting particle system (\ref{zz1.2}) attracts lots of attentions recently.   Gomes and Pavliotis \cite{w} studied the system (\ref{zz1.2}) with coefficients $b(x,\mu,y,\nu) = b(x,\mu,y),$ $\sigma = c$ for which $c$ is a constant. They explored that while the mean field limit and homogenization limit commute over finite time, they do not commute over the long time. Moreover, Delgadino et al.\cite{n} studied the system (\ref{zz1.2}) with $b(x,\mu,y,\nu) = b(y,\nu)$ and $\sigma = c$, and showed that the mean field and homogenization limits do not commute if the mean-field system constrained to the torus undergoes a phase transition, i.e.~multiple steady states exist. Recently, Bezemek and
Spiliopoulos \cite{j,jjj}  established the LDP and moderate deviations of the empirical laws to system (\ref{zz1.2}) with $b(x,\mu,y,\nu) = b(x,\mu,y)$ and $\sigma(x,\mu,y,\nu) = \sigma(x,\mu,y),$ as $N \to \infty$ and $\varepsilon \to 0$ simultaneously.

In this paper, we  aim to study the Freidlin-Wentzell type LDP for the mean-field stochastic equation (\ref{n1}). A related work is the reference \cite{cc} where the authors established the LDP for the following type of multi-scale mean-field stochastic systems
\begin{equation*}\left\{\begin{array}{l}\label{E2}
\displaystyle
d X^{\delta}_t = b(X^{\delta}_t, \mathcal{L}_{X^{\delta}_t}, Y^{\delta}_t)dt+\sqrt{\delta}\sigma(X^{\delta}_t, \mathcal{L}_{X^{\delta}_t})d W^{1}_t, \\
\displaystyle d Y^{\delta}_t =\frac{1}{\varepsilon}f(X^{\delta}_t, \mathcal{L}_{X^{\delta}_t}, Y^{\delta}_t)dt+\frac{1}{\sqrt{\varepsilon}}g( X^{\delta}_t, \mathcal{L}_{X^{\delta}_t}, Y^{\delta}_t)d W^{2}_t,
\end{array}\right.
\end{equation*}
by employing the powerful weak convergence criterion. The weak convergence approach  has been systematically developed by Budhiraja, Dupuis, and Ellis in \cite{d, e, m}. The core  of this method relies on the variational representation formulas for the Laplace transform of bounded continuous functionals and  the equivalence between the LDP and the Laplace principle.
More specifically, the LDP is derived by demonstrating the weak convergence of solutions to the controlled (slow) process  towards its deterministic averaged limit as proved in \cite{cc}. However, due to the dependence of fast process $Y_t^{\delta}$ and its time marginal law $\mathscr{L}_{Y_t^{\delta}}$ in (\ref{n1}),  characterizing the limit  of the controlled slow process $X_t^{\delta,h^{\delta}}$ (see (\ref{2.7}) below)  is more challenging in the weak convergence method. In particular, different from  \cite{cc}, one cannot directly justify the weak convergence criterion in this case.

To deal with this problem, we construct the functional occupation measure corresponding to the controlled fast process $Y_t^{\delta,h^{\delta}}$ (cf.~(\ref{2.7})) and the control variable $h^{\delta}_t $ (cf.~(\ref{2.12}) below). Moreover,
we present the notion of viable pair within the mean-field framework. The definition of viable pair for the classical SDEs was initially introduced in \cite{m} and substantially developed in recent works \cite{x, dd}, which refers to a combination of a trajectory and a measure that captures both the limit averaging dynamics of the controlled slow processes and the invariant measure of the controlled fast process. We highlight that this is an effective method to address the problem because the behavior of the fast component will not converge pathwisely to any specific outcome, but its occupation measure will converge to a limiting measure.

As far as we know, Dupuis and Spiliopoulos \cite{o} studied the LDP for locally periodic SDEs with small noise and fast oscillating coefficients  and achieved significant results using the occupation measure approach. Subsequently, Spiliopoulos \cite{tttt} established the LDP and provide a rigorous mathematical framework for asymptotically efficient importance sampling schemes for stochastic systems with slow and fast dynamics.

However, in comparison to the existing works e.g. \cite{cc,o,tttt}, we consider the mean-field dynamics (\ref{n1}), where the coefficients of (\ref{n1}) not only depend on $X_t^{\delta}$ and its  law $\mathscr{L}_{X_t^{\delta}}$, but also on $Y_t^{\delta}$ and  $\mathscr{L}_{Y_t^{\delta}}$. We will demonstrate that the controlled slow processes and occupation measures $\{(X_t^{\delta,h^{\delta}},\mathbf{P}^{\delta,\Delta})\}_{\delta >0}$, which is defined in (\ref{2.12}), are tight and then have a weakly convergent subsequence. Subsequently, we  establish the  upper and lower bounds
of Laplace principle through the variational representation formulas for functionals of  Brownian motion, then the  LDP is obtained. It is worth to note that the proof of  upper bound of Laplace principle is more complicated than the lower bound, and in this case we construct the control by feedback form and establish a nearly optimal control to achieve the desired  bounds.

The rest of the paper is outlined as follows. In Section 2, we present some notations and definitions of LDP and introduce the main assumptions on coefficients. In Section 3, we outline the weak convergence approach and present the main result. Section 4 is dedicated to deriving some preliminary estimates for the controlled equations. Section 5 is dedicated to proving the LDP. Section 6 is the appendix.

Throughout the work, we use $C$ to denote a generic positive constant, whose specific value may vary in different lines. When necessary, we will specify dependence of the constant on parameters, such as $C_T$.

\section{Main assumptions}
\setcounter{equation}{0}
\setcounter{definition}{0}
We first collect some notations that will be frequently used in this work. We denote the Euclidean vector norm and inner product by $|\cdot|$ and $\langle\cdot,\cdot\rangle$, respectively. Let $\|\cdot\|$ be the Hilbert-Schmidt norm on a matrix, i.e., $\|g\|^2:= Tr(gg^*)$. The tensor product $\mathbb{R}^{n} \otimes \mathbb{R}^{m}$ represents the space of all $n \times m$-dimensional matrix for $n , m\in \mathbb{N}_+$.

Let $\mathcal{P}(\mathbb{R}^{n})$ stand for the space of all probability measures on $(\mathbb{R}^{n},\mathcal{B}(\mathbb{R}^{n}))$. For any $p\geq 1$, we set
$$\mathcal{P}_p(\mathbb{R}^{n}):=\Big\{\mu\in\mathcal{P}(\mathbb{R}^{n}):\mu(|\cdot|^p):=\int_{\mathbb{R}^{n}}|\xi|^{p}\mu(d\xi)<\infty\Big\}.$$
It is known that $\mathcal{P}_p(\mathbb{R}^{n})$ is a Polish space under the $L^p$-Wasserstein distance
$$\mathbb{W}_{p}(\mu,\nu):=\inf_{\pi\in\mathscr{C}(\mu,\nu)}\Big(\int_{\mathbb{R}^{n}\times \mathbb{R}^{n}}|\xi-\eta|^p\pi(d\xi,d\eta)\Big)^{\frac{1}{p}},~\mu,\nu\in\mathcal{P}_p(\mathbb{R}^{n}),$$
where $\mathscr{C}(\mu,\nu)$ stands for the set of all couplings for the measures $\mu$ and $\nu$, i.e., $\pi\in\mathscr{C}(\mu,\nu)$ is a probability measure on $\mathbb{R}^{n}\times \mathbb{R}^{n}$ such that $\pi(\cdot\times \mathbb{R}^{n})=\mu$ and $\pi(\mathbb{R}^{n}\times \cdot)=\nu$.

Let's define the following sets  frequently used in the theory of LDP,
\begin{equation}\nonumber
\mathcal{A}:=\left\{h: h~ \text{is}~\mathbb{R}^{d_1+d_2}\text{-valued}~\mathcal{F}_t\text{-predictable process and}\int_0^T|h_s|^2ds<\infty,\mathbb{P}\text{-a.s.}\right\},
\end{equation}
\begin{equation}\nonumber
S_M:=\left\{h\in L^2([0,T];\mathbb{R}^{d_1+d_2}):\int_0^T|h_s|^2ds\leq M\right\},
\end{equation}
and
\begin{equation}\nonumber
\mathcal{A}_{M}:=\Big\{h\in\mathcal{A}:h(\omega)\in S_{M},\mathbb{P}\text{-a.s.}\Big\}.
\end{equation}

Now we recall the definition of LDP and Laplace principle.  Consider the family of random variables $\{X^\delta\}_{\delta>0}$ defined on a probability space $(\Omega,\mathcal{F},\mathbb{P})$ and taking values in a Polish space $\mathcal{E}$. The rate function of LDP is defined as follows.

\begin{definition}\label{2.1}(Rate function)
A function $I:\mathcal{E} \to [0,+\infty]$ is called a rate function if $I$ is
lower semicontinous. Moreover, a rate function $I$ is called a good rate function if for each constant $K$, the level set $\{x\in \mathcal{E}: I(x) \leq K\}$ is a compact subset of $\mathcal{E}$.
\end{definition}

\begin{definition}\label{2.2}(LDP)
The random variable family $\{X^\delta\}_{\delta>0}$ is said to
satisfy the LDP on $\mathcal{E}$ with rate function $I$ if the following two conditions hold:

(i) (LDP lower bound) For any open set $G \subset \mathcal{E}$,

\begin{equation}\nonumber
	\liminf_{\delta\to0}\delta\log\mathbb{P}(X^{\delta}\in G)\geq-\inf_{x\in G}I(x),
\end{equation}

(ii) (LDP upper bound) For any closed set $ F \subset \mathcal{E}$,
\begin{equation}\nonumber
	\limsup_{\delta\to0}\delta\log\mathbb{P}(X^\delta\in F)\leq-\inf_{x\in F}I(x).
\end{equation}
\end{definition}

\begin{definition}\label{2.3}(Laplace principle)
The sequence  $\{X^\delta\}_{\delta>0}$ is said to be satisfied the Laplace
principle upper bound (respectively, lower bound) on $\mathcal{E}$ with a rate function $I$ if for each bounded continuous real-valued function $\Lambda$ : $\mathcal{E} \to \mathbb{R}$
\begin{equation}\nonumber
\limsup_{\delta\to0}-\delta\log\mathbb{E}\Big\{\exp[-\frac1\delta\Lambda(X^\delta)]\Big\}\leq\inf_{x\in\mathcal{E}}\big(\Lambda(x)+I(x)\big),
\end{equation}
$\Big(respectively,$
\begin{align}\nonumber
\liminf_{\delta\to0}-\delta\log\mathbb{E}\Big\{\exp[-\frac{1}{\delta}\Lambda(X^\delta)]\Big\}\geq\inf_{x\in\mathcal{E}}\big(\Lambda(x)+I(x)\big)\Big).
\end{align}
\end{definition}

It is well-known that if $\mathcal{E}$ is a Polish space and $I$ is a good rate function, then the LDP and Laplace principle are equivalent (see e.g.  \cite{m,ppp,ww}).

In this work, we assume that the maps
$$
\begin{aligned}
	&b:\mathbb{R}^{n}\times\mathcal{P}_{2}(\mathbb{R}^{n})\times\mathbb{R}^{m}\times\mathcal{P}_{2}(\mathbb{R}^{m})\rightarrow\mathbb{R}^{n}; \\
	&\sigma:\mathbb{R}^{n}\times\mathcal{P}_{2}(\mathbb{R}^{n})\times\mathbb{R}^{m}\times\mathcal{P}_{2}(\mathbb{R}^{m})\rightarrow\mathbb{R}^{n}\otimes\mathbb{R}^{d_{1}}; \\
	&f:\mathbb{R}^{m}\rightarrow\mathbb{R}^{m}; \\
	&g:\mathbb{R}^m\to\mathbb{R}^m\otimes\mathbb{R}^{d_2}
\end{aligned}
$$
satisfy the following conditions:

\begin{enumerate}
	\item [$(\mathbf{A_1})$] Suppose that there exist constants $C$, $\kappa$$>0$ such that for all $ x_{1},x_{2}\in\mathbb{R}^{n},\mu_{1},\mu_{2}\in\mathcal{P}_{2}(\mathbb{R}^{n}),\nu_1,\nu_2\in\mathcal{P}_{2}(\mathbb{R}^{m}),y_{1},y_{2}\in\mathbb{R}^{m},$
	\begin{align}\label{a1}
|b(x_1,\mu_1,y_1,\nu_1)&-b(x_2,\mu_2,y_2,\nu_2)|+\|\sigma(x_1,\mu_1,y_1,\nu_1)-\sigma(x_2,\mu_2,y_2,\nu_2)\|\nonumber\\
&\le C\big[|x_1-x_2|+|y_1-y_2|+\mathbb{W}_2(\mu_1,\mu_2)+\mathbb{W}_2(\nu_1,\nu_2)\big],
	\end{align}
	\begin{align}\label{a11}
&|f(y_1)-f(y_2)|+\|g(y_1)-g(y_2)\| \le C|y_1-y_2|.
	\end{align}
	
	Moreover,
	\begin{align}\label{a111}
2\langle f(y_1)-f(y_2),y_1-y_2\rangle+\|g(y_1)-g(y_2)\|^2\le-\kappa|y_1-y_2|^2. 
	\end{align} 

\item [$(\mathbf{A_2})$] Suppose that $g$ is bounded and there exist $c_1,c_2>0$ such that for all $x\in \mathbb{R}^n$, $y\in \mathbb{R} ^m$, $\mu \in \mathcal{P} _2(\mathbb{R}^n)$, $\nu\in \mathcal{P} _2(\mathbb{R}^m)$ and $\xi\in \mathbb{R} ^n, $
\begin{equation}\label{rf1}
c_{1}|\xi|^{2}\leq\langle\sigma\sigma^{*}(x,\mu,y,\nu)\xi,\xi\rangle\leq c_{2}|\xi|^{2}.
\end{equation}
\end{enumerate}

We give some remarks regarding the aforementioned conditions.
\begin{Rem}
(i) Conditions $(\ref{a1})$ and $(\ref{a11})$ are utilized to ensure the existence and uniqueness of strong solutions to system $(\ref{n1})$-$(\ref{n0})$. Conditions $(\ref{a111})$ and the boundedness of $g$ imply that for any $\beta\in(0,\kappa)$, there exists $C_\beta>0$ such that for any $y\in\mathbb{R}^m$,
\begin{equation}\label{2.4}
2\langle f(y),y\rangle+\|g(y)\|^2\leq-\beta|y|^2+C_\beta.
\end{equation} 
Moreover, the dissipative condition $(\ref{a111})$ is also employed to ensure the existence and uniqueness of the invariant measure to Eq.~$(\ref{n2})$ below $($cf.~\cite[Theorem 4.3.9]{ll}$)$.

(ii) Note that in the study of the averaging principle, the assumptions $\langle f(y),y\rangle+\|g(y)\|^2\to-\infty$, as $|y|\to\infty$ (or a slightly stronger version like $\langle f(y),y\rangle+\|g(y)\|^2\leq-c|y|^\alpha$ for $c,\alpha>0$ and all sufficiently large $|y|)$,  together with the  non-degeneracy condition of $gg^*$   are sufficient, which   guarantee the existence and uniqueness of the invariant measure to Eq.~$(\ref{n2})$. However, in the study of the LDP, the stronger assumption $(\ref{2.4})$ is also required to obtain the uniform estimates of the controlled fast process $Y^{\delta,h^\delta}$ defined in the system $(\ref{3.3})$ below.
If we consider the above-mentioned more general case, the dissipation  would be insufficient to dominate  the growth of the controlled dynamics in the analysis of LDP.

(iii) Condition $(\ref{rf1})$ is primarily employed to provide an explicit expression for the rate function $I$ in  LDP.

\end{Rem}

\section{Main result on LDP}
\setcounter{equation}{0}
\setcounter{definition}{0}
In this section, we  first provide a  overview for the weak convergence approach that  is systematically developed in \cite{f,m}. At its core, this approach hinges on two pivotal aspects: the equivalence between the LDP and the Laplace principle, and the employment of variational representation for the exponential functional of Brownian motions (cf. \cite{c, d}). Consequently, we  focus on establishing the Laplace principle instead of directly proving the LDP. Then we present the main results of this work.
\subsection{Weak convergence approach}\label{subsec2.1}
Let $\mathcal{E}$ denote the space of all continuous functions on $\mathbb{R}^n$, i.e., $\mathcal{E}:=C([0,T];{\mathbb{R}^n})$. In this study, our objective is to establish that the family $\{X^\delta\}_{\delta>0}$ satisfies the Laplace principle with speed $\delta$. Specifically, this means that for every bounded and continuous function $\Lambda:\mathcal{E}\to\mathbb{R}$,
\begin{equation}\label{2.5}
	\lim\limits_{\delta\to0}-{\delta}\log\mathbb{E}\left[\exp\left\{-\frac{1}{\delta}\Lambda(X^\delta)\right\}\right]=\inf\limits_{x\in\mathcal{E}}\left[I(x)+\Lambda(x)\right].
\end{equation}
The derivation of the Laplace principle relies on a variational representation for functionals of Brownian motions, which permits us to reformulate the prelimit expression on the left hand side of (\ref{2.5}). Let $F(\cdot)$ be a bounded and measurable real-valued function defined on $C([0,T];{\mathbb{R}^d})$. In light of \cite{c} or \cite{d}, we obtain
$$
-\log\mathbb{E}\Big[\exp\Big\{-F(W)\Big\}\Big]=\inf_{h\in\mathcal{A}}\mathbb{E}\left[\int_{0}^{T}|h_{s}|^{2}ds+F\left(W+\int_{0}^{\cdot}h_{s}ds\right)\right],
$$
where $W$ is a standard $d$-dimensional Brownian motion.

In the current context, we consider $W = (W^1, W^2)$ and $d = d_1 + d_2$. Under assumption $({\mathbf{A_1}})$, where both the coefficients of the slow and fast components exhibit globally Lipschitz continuity, the strong well-posedness is well-known (cf.~e.g.~\cite{xx}). 
\begin{Rem}\label{remarkfro}
For any $\mu \in C([0,T];\mathcal{P}_2(\mathbb{R}^n))$ and $\rho\in C([0,T];\mathcal{P}_2(\mathbb{R}^m))$, which can be viewed as a deterministic measure flow, we consider the following reference SDEs
\begin{equation}\label{refSDE}
d\tilde{X}^{\delta}_t=b^{\mu_t,\rho_t}(\tilde{X}^{\delta}_t, \tilde{Y}^{\delta}_t)dt+\sqrt{\delta}\sigma^{\mu_t,\rho_t}(\tilde{X}^{\delta}_t,\tilde{Y}^{\delta}_t)d W_t^1,~ \tilde{X}^{\delta}_0=x,
\end{equation}
where   we denote $b^{\mu,\rho}(x,y):=b(x,\mu,y,\rho)$, $\sigma^{\mu,\rho}(x,y):=\sigma(x,\mu,y,\rho)$. Note that Eq.~$(\ref{refSDE})$ is a  classical SDE, one can apply the classical Yamada-Watanabe theorem to obtain that there exists a measurable map
$\mathcal{G}_{\mu,\rho}: C([0,T]; \RR^{d_1+d_2})\rightarrow C([0,T]; \RR^n)$ such that we have the representation
\begin{equation}\label{Yawa}
\tilde{X}^{\delta}=\mathcal{G}_{\mu,\rho}\big(W_{\cdot}\big).
\end{equation}
Now we fix $\mu=\mu^{\delta}:=\mathscr{L}_{X^{\delta}}$ and $\rho=\rho^{\delta}:=\mathscr{L}_{Y^{\delta}}$, then the original slow equation $(\ref{n1})$ reduces to
\begin{equation}
d X^{\delta}_t=b^{\mu_t,\rho_t}(X^{\delta}_t, Y^{\delta}_t)dt+\sqrt{\delta}\sigma^{\mu_t,\rho_t}(X^{\delta}_t, Y^{\delta}_t)d W_t^1,~ X^{\delta}_0=x.
\end{equation}
It is clear  that $X^{\delta}_t$ is also a solution of Eq.~$(\ref{refSDE})$ with $\mu=\mu^{\delta},\rho=\rho^{\delta}$, then
by the strong uniqueness of  Eq.~$(\ref{refSDE})$, we know
$$X^{\delta}_t=\tilde{X}^{\delta}_t,~t\in[0,T].$$
Therefore, it follows from the representation $(\ref{Yawa})$ that 
$$X^{\delta}=\mathcal{G}_{\mu,\rho}\big(W_{\cdot}\big)=\mathcal{G}_{\mu^{\delta},\rho^{\delta}}\big(W_{\cdot}\big).$$
\end{Rem}

From Remark \ref{remarkfro}, for simplicity of notation, we denote $\mathcal{G}^{\delta}:=\mathcal{G}_{\mu^{\delta},\rho^{\delta}}$.
For any $h^\delta\in\mathcal{A}_M$, let us define
\begin{equation}\label{2.6}
	X^{\delta,h^{\delta}}:=\mathcal{G}^{\delta}\Big(W.+\frac{1}{\sqrt{\delta}}\int_{0}^{\cdot}h_{s}^{\delta}ds\Big).
\end{equation}
It is the solution corresponding to the following stochastic control problem
\begin{equation}\label{2.7}
	\left.\left\{\begin{aligned}
		dX_t^{\delta,h^{\delta}}=&b(X_t^{\delta,h^{\delta}},\mathscr{L}_{X_t^{\delta}},Y_t^{\delta,h^{\delta}},\mathscr{L}_{Y_t^{\delta}})dt+\sigma(X_t^{\delta,h^{\delta}},\mathscr{L}_{X_t^{\delta}},Y_t^{\delta,h^{\delta}},\mathscr{L}_{Y_t^{\delta}})h_t^{1,\delta}dt\\&+\sqrt{\delta}\sigma(X_t^{\delta,h^{\delta}},\mathscr{L}_{X_t^{\delta}},Y_t^{\delta,h^{\delta}},\mathscr{L}_{Y_t^{\delta}})dW_t^1,
\\dY_t^{\delta,h^{\delta}}=&\frac{1}{\varepsilon}f(Y_t^{\delta,h^{\delta}})dt+\frac{1}{\sqrt{\delta\varepsilon}}g(Y_t^{\delta,h^{\delta}})h_t^{2,\delta}dt+\frac{1}{\sqrt{\varepsilon}}g(Y_t^{\delta,h^{\delta}})dW_t^2,
\\X_0^{\delta,h^{\delta}}=&x,Y_0^{\delta,h^{\delta}}=y,
	\end{aligned}\right.\right.
\end{equation}
where the controls $h_t^{1,\delta}:=P_1h_t^\delta$, $h_t^{2,\delta}:=P_2h_t^\delta$.

\vspace{1mm}
After setting $F(W)=\frac{1}\delta\Lambda(X^\delta)$ and rescaling the controls by $\sqrt{\delta}$, we attain the following representation
\begin{equation*}
	-\delta\log\mathbb{E}\left[\exp\left\{-\frac{1}{\delta}\Lambda(X^\delta)\right\}\right] =\inf_{h\in\mathcal{A}}\mathbb{E}\left[\frac{1}{2}\int_0^T|h_s|^2ds+\Lambda(X^{\delta,h})\right],
\end{equation*}
where $X^{\delta,h}$ is defined by (\ref{2.6}) with $h$ replacing $h^\mathrm{\delta}$. Hence, we need  to investigate the limiting behaviors of the controlled  process $X^{\delta,h^\delta}$ in the  weak convergence approach.

Due to the result of  the averaging principle (cf.~\cite[Theorem 2.1]{vvv}), let $\delta \to 0$ (hence $\varepsilon \to 0$) in (\ref{n1}), we can get the following ordinary differential equation
\begin{equation}\label{2.9}
	\frac{d\bar{X}_t}{dt}=\bar{b}(\bar{X}_t,\mathscr{L}_{\bar{X}_t}),~~\bar{X}_0=x,
\end{equation}
where the new coefficient $\bar{b}$ is defined by
\begin{equation*}\label{2.10}
	\bar{b}(x,\mu)=\int_{\mathbb{R}^m}b^{\nu}(x,\mu,y)\nu(dy).
\end{equation*}
By \cite[Theorem 2.1]{vvv}, we obtain that $X^\delta$ converges to  $\bar{X}$ in the $L^2(\Omega;C([0,T];\mathbb{R}^n))$-sense,  i.e.,
\begin{equation}\label{z3.1}
	\mathbb{E}\Big[\sup\limits_{t\in[0,T]}|X_t^\delta-\bar{X}_t|^2\Big] \to 0,~~ \text{as} ~\delta \to 0.
	\end{equation}
Here, we denote $b^{\nu}(x,\mu,y) := b(x, \mu, y, \nu)$, and $\nu$ denotes the unique invariant measure of the process $Y_t$ satisfying the following equation
\begin{equation}\label{n2}
		dY_t=f(Y_t)dt+g(Y_t)d\tilde{W}_t^2,~~Y_0=y,
\end{equation}
where $\{\tilde{W}_t^2\}_{t\geq0}$ is a $d_2$-dimensional Brownian motion on complete filtered probability space $({\Omega},{\mathcal{F}},({\mathcal{F}}_t)_{t\geq 0},{\mathbb{P}}).$

It is evident that (\ref{2.9}) has a unique solution, denoted throughout this work as $\bar{X}$, which satisfies $\bar{X} \in C([0,T];\mathbb{R}^n)$. We also remark that the solution $\bar{X}$ of (\ref{2.9}) is a deterministic path, and its law $\mathscr{L} _{\bar{X} _{t}}= \delta_{\bar{X} _{t}},$ where $\delta_{\bar{X} _{t}}$ is the Dirac measure of $\bar{X}_{t}.$

\subsection{Main result}
 In order to present the main LDP result, we first introduce some additional notations and definitions.

 Let $d:=d_1 + d_2$ and use $\mathbb{R}^d := \mathbb{R}^{d_1} \times \mathbb{R}^{d_2}$ to denote the space in which the control $h$ takes values. In addition, we denote  $h^{1}:=P_1 h,~h^{2}:=P_2 h$, and  $\mathcal{Y}:=\mathbb{R}^m$ to emphasize the state space of the fast component. Let $A_1,A_2,A_3$ be Borel subsets in $\mathbb{R}^d, \mathcal{Y}, [0,T]$, respectively. Let $\Delta:=\Delta(\delta)$ be a time-scale separation, which satisfies
\begin{align}\label{2.11}
	\Delta(\delta)\to0,~\frac{\varepsilon}{\delta \Delta} \to 0,  ~\mathrm{as~} \delta\to0.
\end{align}

Concerning the joint perturbations of the control and the controlled fast process in the multi-scale system (\ref{2.7}), it is proper to introduce the following occupation measure
\begin{equation}\label{2.12}
	\mathbf{P}^{\delta,\Delta}(A_1\times A_2\times A_3):=\int_{A_3}\dfrac{1}{\Delta}\int_t^{t+\Delta}\mathbf{1}_{A_1}(h_s^{\delta})\mathbf{1}_{A_2}(Y_s^{\delta,h^{\delta}})dsdt,
\end{equation}
which captures the joint limit behaviour of $h^{\delta}$  and $Y^{\delta,h^{\delta}}$. Throughout the work, we adopt the convention that the control
\begin{equation}\label{e04}
h_t= h_t^\delta= 0~~~~~~\text{if}~~~~~~~ t> T.
\end{equation}
 We also mention that for any bounded continuous functions $\psi$, we have
\begin{align}\label{2.13}
	&\int_{\mathbb{R}^d\times\mathcal{Y}\times[0,T]}\psi(X_{t}^{\delta,h^{\delta}},\mathscr{L}_{X_{t}^{\delta}},y,\mathscr{L}_{Y_{t}^{\delta}},h)\mathbf{P}^{\delta,\Delta}(dhdydt) \nonumber\\
	=&\int_{0}^{T}\frac{1}{\Delta}\int_{t}^{t+\Delta}\psi(X_{t}^{\delta,h^{\delta}},\mathscr{L}_{X_{t}^{\delta}},Y_{s}^{\delta,h^{\delta}},\mathscr{L}_{Y_{t}^{\delta}},h_{s}^{\delta})dsdt.
\end{align}

Define a map $\Phi:\mathbb{R}^n\times\mathcal{P}_2(\mathbb{R}^n)\times\mathcal{Y}\times\mathcal{P}_2(\mathbb{R}^m)\times\mathbb{R}^d\to\mathbb{R}^n$ by
\begin{equation}\label{2.14}
	 \Phi(x,\mu,y,\nu,h):=\bar{b}(x,\mu)+\sigma(x,\mu,y,\nu)P_1h.
\end{equation}

Recall  that $\nu$ is the unique invariant measure associated with $(\ref{n2})$. In what follows, we recall the concept of viable pair in the mean-field version, which effectively characterizes the limits of the controlled slow-fast systems.
\begin{definition}$($Viable pair$)$\label{d2.4}
	 A pair $(\varphi,\mathbf{P})\in $ $C([0,T];\mathbb{R}^n)\times\mathcal{P}(\mathbb{R}^d\times\mathcal{Y}\times[0,T])$ is said to be  viable w.r.t.~$(\Phi,\nu,x, \bar{X} )$  and we write $( \varphi, \mathbf{P}) \in \mathcal{V} _{(\Phi,\nu,x, \bar{X} )}$, if the following statements hold:
	
	$(i)$ The measure $\mathbf{P}$ has finite second moments, i.e.,
	
	\begin{equation*}\label{2.15}
		\int_{\mathbb{R}^d\times\mathcal{Y}\times[0,T]}\big[|h|^{2}+|y|^{2}\big]\mathbf{P}(dhdyds)<\infty.
	\end{equation*}
	
	$(ii)$ For all $t\in [ 0, T] , $
	
	\begin{equation}\label{2.16}
	\varphi_{t}=x+\int_{\mathbb{R}^d\times\mathcal Y\times[0,t]}\Phi(\varphi_{s},\mathscr{L}_{\bar{X}_{s}},y,\nu,h)\mathbf{P}(dhdyds).
	\end{equation}
	
	$(iii)$ For all $A_1\times A_2\times A_3\in\mathcal{B}(\mathbb{R}^d\times\mathcal{Y}\times[0,T])$,
	
	\begin{equation}\label{2.17}
		\mathbf{P}(A_{1}\times A_{2}\times A_{3})=\int_{A_{3}}\int_{A_{2}}\eta(A_{1}|y,t)\nu(dy)dt,	
	\end{equation}
	where $\eta$ is a stochastic kernel $($cf.~\cite[Appendix B.2]{f}$)$ given $(y,t)\in\mathcal{Y}\times[0,T]$. In particular, this implies that the last marginal of $\mathbf{P}$ is the Lebesgue measure on $[0,T]$, i.e., for all $t\in[0,T]$,
	\begin{equation}\label{2.18}
		\mathbf{P}(\mathbb{R}^d\times\mathcal{Y}\times[0,t])=t.
	\end{equation}
\end{definition}

The following is the main result in this work.
\begin{theorem}\label{t2.1}
	Suppose that the assumptions $(\mathbf{A_1})$-$(\mathbf{A_2})$ hold and the scale $\varepsilon=\varepsilon(\delta)$ satisfies $\lim_{\delta\to0} \varepsilon/\delta = 0$.  Then $\{X^\delta\}_{\delta>0}$ satisfies the LDP with the good rate function $I$ given by
	\begin{equation}\label{sulv}
		I(\varphi):=\inf_{(\varphi,\mathbf{P})\in\mathcal{V}_{(\Phi,\nu,x, \bar{X} )}}\left\{\frac{1}{2}\int_{\mathbb{R}^d\times\mathcal{Y}\times[0,T]}|h|^{2} \mathbf{P}(dhdydt)\right\}
	\end{equation}
	with the convention that the infimum over the empty set is $\infty$, where $\Phi$ is defined by $(\ref{2.14})$ and  $\bar{X}$ is the solution of $(\ref{2.9})$.

Furthermore, the rate function $I$ has the following explicit representation
\begin{align*} I(\varphi)=\begin{cases}\frac{1}{2}\int_0^T|Q^{-1/2}(\varphi_{t},\mathscr{L}_{\bar{X}_t},\nu)(\dot{\varphi}_t-\bar{b}(\varphi_t,\mathscr{L}_{\bar{X}_t}))|^2dt,&\varphi(0)=x,\varphi ~\text{is absolutely continuous},\\+\infty,&\text{otherwise},\end{cases}
\end{align*}
where
\begin{align}\label{z2.21}
	Q(\varphi_{t},\mathscr{L}_{\bar{X}_{t}},\nu):=\int_{\mathcal{Y}}\sigma(\varphi_{t},\mathscr{L}_{\bar{X}_{t}},y,\nu) P_1P_1^{*}\sigma^{*}(\varphi_{t},\mathscr{L}_{\bar{X}_{t}},y,\nu)\nu(dy).
\end{align}
\end{theorem}

\begin{Rem}
(i) According to $(\mathbf{A_2})$, it is clear that $\sigma \sigma^*$ is a positive definite matrix. Since $P_1:\mathbb{R}^{d_1 + d_2} \to \mathbb{R}^{d_1}$ is a projection mapping, it follows directly that $(\sigma P_1)(\sigma P_1)^{*}$ is uniformly positive definite, and consequently so is $Q$.

(ii) In fact, the non-degeneracy of $\sigma\sigma^*$ is  required only to transfer the control into the feedback form $\bar{h}(y)$ and derive the explicit rate function representation in the proof of the Laplace upper bound, see Subsection \ref{sec5.4} for details. The proof of the Laplace lower bound does not require the non-degeneracy of $\sigma\sigma^*$, which holds also   for degenerate noise. 
\end{Rem}

\begin{Rem}
(i)	The LDP for multi-scale mean-field SDEs were initially studied in \cite{cc} by employing the weak convergence criterion directly, where only the laws of slow process are considered in the system. Recently, the authors in \cite{vv, yy} established the LDP for multi-scale mean-field SDEs driven by fractional noise.

However, all the existing works \cite{cc, vv, yy} do not allow the cases that the system depends on the laws of fast process $($i.e.~$\mathscr{L}_{Y_t^{\delta}}$$)$ and the diffusion coefficient $\sigma$ in the slow component depends on the fast component. In the present work, we derive the LDP for a more general system $(\ref{n1})$ by utilizing the functional occupation measure and constructing the controls of feedback form, rather than  employing the weak convergence criterion directly.

(ii) We also note that there are two reasons for restricting the coefficients $f$ and $g$ in the fast component to depend only on $Y_t^\delta$. 
Firstly, if $f$ and $g$ depend on the law $\mathscr{L}_{Y_t^{\delta}}$, the frozen equation associated to the fast process would become a McKean-Vlasov SDE, i.e.
\begin{equation}\label{eq0}
dY_t=f(Y_t,\mathscr{L}_{Y_t})dt+g(Y_t,\mathscr{L}_{Y_t})d\tilde{W}_t^2,~~Y_0=y.
\end{equation}
In this case, the operator 
\begin{equation}\label{eq1}
P_tf(y):=\mathbb{E}f(X_t^y)
\end{equation}
is no longer a semigroup, as stated in  \cite{xx}. However, the   time discretization method is employed in our proof, specifically, in the proof of Lemma \ref{l4.2} and in establishing the upper bound of the Laplace principle, where the semigroup and  Markov properties of the operator  $(\ref{eq1})$ are essentially required. An alternative observation is the frozen equation $(\ref{eq0})$ can be transformed into a non-autonomous classical SDE
\begin{equation*}
dY_t=\tilde{f}(t,Y_t)dt+\tilde{g}(t,Y_t)d\tilde{W}_t^2,~~Y_0=y,
\end{equation*}
where $\tilde{f}(t,Y_t):=f(Y_t,\mathscr{L}_{Y_t}), \tilde{g}(t,Y_t):=g(Y_t,\mathscr{L}_{Y_t})$. When employing the time discretization method, the time-homogeneous property of the associated semigroup is also required.

On the other hand, the coefficients $f,g$ are independent of the slow process $X_t^{\delta}$, which corresponds to the non-fully coupled case in multi-scale systems due to technical constraints. In fact, for the fully coupled case, significant difficulties arise in proving the averaging principle. A  recent work \cite{HLX} introduced a novel argument and employed a non-autonomous approximation method to address this issue.
It is of interest to  investigate whether the time discretization method can be  adapted to handle the dependence on the law $\mathscr{L}_{Y_t^{\delta}}$, and to explore whether the non-autonomous approximation method proposed in \cite{HLX} can be extended to study the LDP for fully coupled
multi-scale mean-field systems. These problems are left for future work. 
\end{Rem}

\begin{Rem}
In general, one can investigate the asymptotic behavior in all possible interaction regimes, i.e.,
\begin{equation*}
\lim_{\varepsilon\rightarrow0}\frac{\varepsilon}{\delta}=\left\{
  \begin{array}{ll}
    0, & \hbox{Regime 1;} \\
    \gamma, & \hbox{Regime 2;} \\
    \infty, & \hbox{Regime 3.}
  \end{array}
\right.
\end{equation*}
In the present work, we focus on Regime 1, which  allows for the decoupling of the invariant measure $\nu$ and the control $h$ from the limiting occupation measures. This plays an important role in the proof of Laplace principle upper bound. We  point out that our method is also applicable in proving the Laplace principle lower bound in Regime 2, which does not depend on such decoupling.

In Regimes 2 and 3,  the proof of Laplace principle upper bound becomes significantly more involved.
Specifically,  in  Regimes 2 and 3, the invariant measure $\nu^{h}$ associated with the frozen controlled fast dynamics  depends on the control variable $h$. This dependence on $h$ substantially complicates the analysis; in particular, it is unclear whether $\nu^{h}$ is an invariant measure of certain process or not. 
\end{Rem}

\section{Preliminaries}
\setcounter{equation}{0}
\setcounter{definition}{0}

\subsection{Some a priori estimates}
In this section, we present several a priori estimates of the controlled processes. These estimates will be frequently utilized in proving the main result.

The following are the estimates of solutions $(X^{\delta,h^{\delta}},Y^{\delta,h^{\delta}})$ to the control problem (\ref{2.7}) .

\begin{lemma}\label{ll3.1}
For any $\{h^{\delta}\}_{\delta>0}\subset\mathcal{A}_{M}$, there exists a constant $C_{M,T}>0$ which is independent of $\delta$ such that
	\begin{equation}\label{3.1}
		\mathbb{E}\Big[\sup\limits_{t\in[0,T]}|X_t^{\delta,h^\delta}|^4\Big]\leq C_{M,T}(1+|x|^4+|y|^4),
	\end{equation}
	and for any $p \geq 1$, there exists $C_{p,M,T}>0$ such that
	\begin{equation}\label{3.2}
		\mathbb{E}\Bigg[\Bigg(\int_0^T|Y_t^{\delta,h^\delta}|^2dt\Bigg)^p\Bigg]\leq C_{p,M,T}(1+|y|^{2p}).
	\end{equation}
\end{lemma}
\begin{proof}
Using It\^o's formula for $|Y_{t}^{\delta,h^{\delta}}|^{2}$, we have
	\begin{align}\label{3.3}
		|Y_{t}^{\delta,h^{\delta}}|^{2} =&|y|^{2}+\frac{1}{\varepsilon}\int_{0}^{t}\Big[2\langle f(Y_{s}^{\delta,h^{\delta}}),Y_{s}^{\delta,h^{\delta}}\rangle+\|g(Y_{s}^{\delta,h^{\delta}})\|^{2}\Big]ds \nonumber \\
		&+\frac{2}{\sqrt{\delta\varepsilon}}\int_{0}^{t}\langle g(Y_{s}^{\delta,h^{\delta}})h_{s}^{2,\delta},Y_{s}^{\delta,h^{\delta}}\rangle ds+M_{t},
	\end{align}
	where $M_t$ is a local martingale given by
	\begin{equation}\nonumber
		M_t:=\frac{2}{\sqrt{\varepsilon}}\int_0^t\langle Y_s^{\delta,h^\delta},g(Y_s^{\delta,h^\delta})dW_s^2\rangle.
	\end{equation}

	By $({\mathbf{A_2}})$, we also have
	\begin{align}\label{3.4}
		\frac{2}{\sqrt{\delta\varepsilon}}\langle g(Y_{s}^{\delta,h^{\delta}})h_{s}^{2,\delta},Y_{s}^{\delta,h^{\delta}}\rangle \leq&\frac{C}{\sqrt{\delta\varepsilon}}|h_{s}^{2,\delta}||Y_{s}^{\delta,h^{\delta}}| \nonumber \\
		\leq&\frac{C}{\delta}|h_{s}^{2,\delta}|^{2}+\frac{\tilde{\beta}}{\varepsilon}|Y_{s}^{\delta,h^{\delta}}|^{2},
	\end{align}
	where in the last step we applied Young's inequality with a small constant $\tilde{\beta}\in(0,\beta)$ in which $\beta$ is defined in (\ref{2.4}).
	
	Combining (\ref{2.4}) and (\ref{3.3})-(\ref{3.4}), it follows that
	\begin{align}\nonumber
		\frac{\kappa_0}{\varepsilon}\int_0^T|Y_s^{\delta,h^\delta}|^2ds\leq|y|^2+\frac{C_T}{\varepsilon}+\frac{C}{\delta}\int_0^T|h_s^{2,\delta}|^2ds+\sup_{t\in[0,T]}|M_t|,
	\end{align}
	where $\kappa_0:=\beta-\tilde{\beta}>0$. Then
	\begin{align}\nonumber
		\mathbb{E}\Bigg[\Bigg(\int_{0}^{T}|Y_{s}^{\delta,h^{\delta}}|^{2}ds\Bigg)^{p}\Bigg] \leq& C_{p}\varepsilon^{p}|y|^{2p}+C_{p,T}+{\frac{C_{p}\varepsilon^{p}}{\delta^{p}}}\mathbb{E}\Bigg[\Bigg(\int_{0}^{T}|h_{s}^{2,\delta}|^{2}ds\Bigg)^{p}\Bigg]+C_{p}\varepsilon^{p}\mathbb{E}\Big[\Big(\sup_{t\in[0,T]}|M_{t}|\Big)^{p}\Big], \\
		\leq&\frac{1}{2}\mathbb{E}\Bigg[\Bigg(\int_{0}^{T}|Y_{s}^{\delta,h^{\delta}}|^{2}ds\Bigg)^{p}\Bigg]+C_{p,T}(1+|y|^{2p})+\frac{C_{p,M,T}\varepsilon^{p}}{\delta^{p}}, \nonumber
	\end{align}
	where we utilized the fact that $ h^{\delta}\in\mathcal{A}_{M}$ and the following estimate in the second step
	\begin{align}
		C_{p}\varepsilon^{p}\mathbb{E}\Big[\Big(\operatorname*{sup}_{t\in[0,T]}|M_{t}|\Big)^{p}\Big]\leq& C_{p}\varepsilon^{\frac{p}{2}}\mathbb{E}\Bigg[\Bigg(\int_{0}^{T}|Y_{s}^{\delta,h^{\delta}}|^{2}ds\Bigg)^{p/2}\Bigg] \nonumber \\
		\leq&\frac{1}{2}\mathbb{E}\Bigg[\Bigg(\int_{0}^{T}|Y_{s}^{\delta,h^{\delta}}|^{2}ds\Bigg)^{p}\Bigg]+C_{p}.\nonumber
	\end{align}
Due to the condition $\lim_{\delta\to0} \varepsilon/\delta = 0$, without loss of generality, we can assume $\frac{\varepsilon}{\delta} < 1$. Thus, (\ref{3.2}) holds.

We proceed to show (\ref{3.1}).
	Firstly, we have
	\begin{align}\label{3.5}
		|X_{t}^{\delta,h^{\delta}}|^{4}
		\leq&|x|^{4}+C\Big|\int_{0}^{t} b(X_{s}^{\delta,h^{\delta}},\mathscr{L}_{X_{s}^{\delta}},Y_{s}^{\delta,h^{\delta}},\mathscr{L}_{Y_{s}^{\delta}}) ds\Big|^4
		\nonumber \\
		&+C\Big|\int_{0}^{t}\sigma(X_{s}^{\delta,h^{\delta}},\mathscr{L}_{X_{s}^{\delta}},Y_{s}^{\delta,h^{\delta}},\mathscr{L}_{Y_{s}^{\delta}})h_{s}^{1,\delta}ds\Big|^4
		\nonumber \\ &
		+C\delta^2\Big|\int_{0}^{t}\sigma(X_{s}^{\delta,h^{\delta}},\mathscr{L}_{X_{s}^{\delta}},Y_{s}^{\delta,h^{\delta}},\mathscr{L}_{Y_{s}^{\delta}})dW_{s}^{1}\Big|^4 \nonumber\\
		=&: |x|^{4}+(\text{I})+(\text{II})+(\text{III}).
	\end{align}
	Due to the condition $(\mathbf{A_1})$, we have
	\begin{align}\label{3.6}
		(\text{I}) \leq C_T& \Bigg(\int_{0}^{T} |b(X_{s}^{\delta,h^{\delta}},\mathscr{L}_{X_{s}^{\delta}},Y_{s}^{\delta,h^{\delta}},\mathscr{L}_{Y_{s}^{\delta}})|^2ds\Bigg)^2 \nonumber\\
		\leq C_T& + C_T\int_{0}^{T}|X_{s}^{\delta,h^{\delta}}|^{4}ds+ C_T\bigg(\int_{0}^{T}\big(\mathbb{E}|X_{s}^{\delta}|^{2}+\mathbb{E}|Y_{s}^{\delta}|^{2}\big)ds\bigg)^2+ C_T\Bigg(\int_{0}^{T}|Y_{s}^{\delta,h^{\delta}}|^{2}ds\Bigg)^2.
	\end{align}
	By condition $(\mathbf{A_2})$ and H\"{o}lder's inequality we derive the following estimate for the third term on the right-hand side of (\ref{3.5})
	\begin{align}\label{3.7}
		(\text{II}) &\leq C_T\Bigg(\int_{0}^{T}\|\sigma(X_{s}^{\delta,h^{\delta}},\mathscr{L}_{X_{s}^{\delta}},Y_{s}^{\delta,h^{\delta}},\mathscr{L}_{Y_{s}^{\delta}})\|^2\cdot|h_{s}^{1,\delta}|^2ds\Bigg)^{2} \nonumber \\
		&\leq C_T\Bigg(\int_{0}^{T}|h_{s}^{1,\delta}|^2ds\Bigg)^{2} \nonumber \\
		&\leq C_{M,T},
	\end{align}
	where we used the fact $\{h^{\delta}\}_{\delta>0}\subset\mathcal{A}_{M}$  in the last step.
	
	By substituting (\ref{3.6})-(\ref{3.7}) into (\ref{3.5}) and taking expectation, we obtain
	\begin{align}\label{es7}
		\mathbb{E}\Big[\sup_{t\in[0,T]}|X_{t}^{\delta,h^{\delta}}|^{4}\Big]
		&\leq C_{M,T} (1+|x|^{4})+ C_{T} \mathbb{E}\int_{0}^{T}|X_{s}^{\delta,h^{\delta}}|^{4}ds+ C_T \bigg(\int_{0}^{T}\big(\mathbb{E}|X_{s}^{\delta}|^{2}+\mathbb{E}|Y_{s}^{\delta}|^{2}\big)ds\bigg)^2\nonumber \\
		&+C_{T} \mathbb{E}\left(\int_{0}^{T}|Y_{t}^{\delta,h^{\delta}}|^{2}dt\right)^2+C \mathbb{E}\Big[\sup_{t\in[0,T]}(\text{III})\Big].
	\end{align}
	Utilizing Burkholder-Davis-Gundy's inequality and $(\mathbf{A_2})$, we have
	\begin{align}
		\mathbb{E}\Big[\sup_{t\in[0,T]}(\text{III})\Big]&\leq \delta^2\mathbb{E}\Bigg(\int_{0}^{T}\|\sigma(X_{s}^{\delta,h^{\delta}},\mathscr{L}_{X_{s}^{\delta}},Y_{s}^{\delta,h^{\delta}},\mathscr{L}_{Y_{s}^{\delta}})\|^2ds\Bigg)^2
		\nonumber \\
		& \leq  C_{T} \delta^2.
	\end{align}
	Following from \cite[Lemma 2.2]{vvv} that we can get the following uniform estimates 
	\begin{align}
		\mathbb{E}\Big[\sup\limits_{t\in[0,T]}|X_t^\delta|^2\Big]\leq& C_T(1+|x|^2+|y|^2),\label{e02}\\
		\sup_{t\in[0,T]}\mathbb{E}|Y_{t}^{\delta}|^{2}\leq& C_{T}(1+|x|^{2}+|y|^{2}).\label{e03}
	\end{align}
	Collecting (\ref{es7})-(\ref{e03}) and utilizing Gronwall's lemma, we obtain
	\begin{equation}\nonumber
		\mathbb{E}\Big[\sup\limits_{t\in[0,T]}|X_t^{\delta,h^\delta}|^4\Big]\leq C_{M,T}(1+|x|^4+|y|^4).
	\end{equation}
	The proof is complete.
\end{proof}

The following lemma provides a time H{\"o}lder continuity estimate for the controlled process $X_{t}^{\delta,h^{\delta}}$, which plays an important role for demonstrating the existence of viable pairs later in the proof.

\begin{lemma}\label{l3.3}
There exists $C_{T} > 0$ such that for any $0\leq t\leq t+\Delta\leq T,$
	\begin{equation}\label{3.12}
		\mathbb{E}|X_{t+\Delta}^{\delta,h^\delta}-X_t^{\delta,h^\delta}|^4\leq C_{M,T}(1+|x|^4+|y|^4)\Delta^2.
	\end{equation}
\end{lemma}
\begin{proof}
Due to (\ref{2.7}), we have
	\begin{align}\label{3.13}
		& \mathbb{E}|X_{t+\Delta}^{\delta,h^{\delta}}-X_{t}^{\delta,h^{\delta}}|^{4}
\nonumber \\
\leq& C\mathbb{E}\bigg|\int_{t}^{t+\Delta}b(X_{s}^{\delta,h^{\delta}},\mathscr{L}_{X_{s}^{\delta}},Y_{s}^{\delta,h^{\delta}},\mathscr{L}_{Y_{s}^{\delta}})ds\bigg|^{4} \nonumber \\
		&+C\mathbb{E}\Big|\int_{t}^{t+\Delta}\sigma(X_{s}^{\delta,h^{\delta}},\mathscr{L}_{X_{s}^{\delta}},Y_{s}^{\delta,h^{\delta}},\mathscr{L}_{Y_{s}^{\delta}})h_{s}^{1,\delta}ds\Big|^{4} \nonumber \\
		&+C\delta^2\mathbb{E}\Big|\int_{t}^{t+\Delta}\sigma(X_{s}^{\delta,h^{\delta}},\mathscr{L}_{X_{s}^{\delta}},Y_{s}^{\delta,h^{\delta}},\mathscr{L}_{Y_{s}^{\delta}})dW_{s}^{1}\Big|^{4} \nonumber \\
		=:&\mathscr{J}_1+\mathscr{J}_2+\mathscr{J}_3.
	\end{align}
	As for the term $\mathscr{J}_1$, it follows from  (\ref{a1}), (\ref{3.1}), (\ref{3.2}) and H{\"o}lder's inequality that
	\begin{align}\label{3.14}
		\mathscr{J}_1\leq&C\Delta^2\mathbb{E}\Big(\int_{t}^{t+\Delta}|b(X_{s}^{\delta,h^{\delta}},\mathscr{L}_{X_{s}^{\delta}},Y_{s}^{\delta,h^{\delta}},\mathscr{L}_{Y_{s}^{\delta}})|^{2}ds\Big)^2 \nonumber \\	\leq&C\Delta^2\bigg(1+\mathbb{E}\int_{t}^{t+\Delta}|X_{s}^{\delta,h^{\delta}}|^{4}ds+\Big(\int_{t}^{t+\Delta}\mathbb{E}|X_{s}^{\delta}|^{2}ds\Big)^2
\nonumber \\	
&+\mathbb{E}\Big(\int_{t}^{t+\Delta}|Y_{s}^{\delta,h^{\delta}}|^{2}ds\Big)^2+\Big(\int_{t}^{t+\Delta}\mathbb{E}|Y_{s}^{\delta}|^{2}ds\Big)^2\bigg)\nonumber \\
		\leq& C_{M,T}(1+|x|^{4}+|y|^{4})\Delta^2.
	\end{align}
	As for the term $\mathscr{J}_2$, due to the condition $(\mathbf{A_2})$ we have
	\begin{align}\label{3.15}
		\mathscr{J}_2\leq&C\mathbb{E}\left(\int_{t}^{t+\Delta}\|\sigma(X_{s}^{\delta,h^{\delta}},\mathscr{L}_{X_{s}^{\delta}},Y_{s}^{\delta,h^{\delta}},\mathscr{L}_{Y_{s}^{\delta}})\|^{2}ds\cdot\int_{0}^{T}|h_{s}^{1,\delta}|^{2}ds\right)^2  \nonumber\\
		\leq& C_{M,T}\Delta^2.
	\end{align}
	Similarly, as for the term $\mathscr{J}_3$, by Burkholder-Davis-Gundy's inequality it follows that
	\begin{equation}\label{zzz3.13}
		\mathscr{J}_3\leq C_{M,T}\delta^2\Delta^2.
	\end{equation}
	Combining (\ref{3.13})-(\ref{zzz3.13}) implies (\ref{3.12}) holds. We complete the proof.
\end{proof}

\subsection{Tightness of controlled processes}
In the following proposition, we prove the tightness of the family $\{(X^{\delta,h^{\delta}},\mathcal{R}_1^\delta,\mathcal{R}_2^\delta,\mathbf{P}^{\delta,\Delta})\}_{\delta>0}$
in
$$\mathcal{Z}_T:=C([0,T];\mathbb{R}^n)^{\otimes3}\times \mathscr{P}(\mathbb{R}^d\times\mathcal{Y}\times[0,T]), $$
\noindent where $C([0,T];\mathbb{R}^n)^{\otimes3}:=C([0,T];\mathbb{R}^n)\times C([0,T];\mathbb{R}^n)\times C([0,T];\mathbb{R}^n)$, and we also prove the uniform integrability  of the family $\{\mathbf{P}^{\delta,\Delta}\}_{\delta>0}$. 
	
	For reader's convenience, we recall
	\begin{align}\label{4.1}
		X_t^{\delta,h^{\delta}}=&~x+\int_{0}^{t}b(X_s^{\delta,h^{\delta}},\mathscr{L}_{X_s^{\delta}},Y_s^{\delta,h^{\delta}},\mathscr{L}_{Y_s^{\delta}})ds+\int_{0}^{t}\sigma(X_s^{\delta,h^{\delta}},\mathscr{L}_{X_s^{\delta}},Y_s^{\delta,h^{\delta}},\mathscr{L}_{Y_s^{\delta}})h_s^{1,\delta}ds \nonumber\\
		&+\sqrt{\delta}\int_{0}^{t}\sigma(X_s^{\delta,h^{\delta}},\mathscr{L}_{X_s^{\delta}},Y_s^{\delta,h^{\delta}},\mathscr{L}_{Y_s^{\delta}})dW_s^1 \nonumber\\
		=:&~x+\sum_{i=1}^3\mathcal{R}_i^\delta(t).
	\end{align}

\begin{proposition}\label{p4.1}
Fix $M<\infty$. Suppose $\{h^{\delta}\}_{\delta>0}\subset\mathcal{A}_{M}$, we have

$(i)$ the family $\{(X^{\delta,h^{\delta}},\mathcal{R}_1^\delta,\mathcal{R}_2^\delta,\mathbf{P}^{\delta,\Delta})\}_{\delta>0}$  is tight;

$(ii)$ define the set
\begin{equation}\nonumber
\mathcal{U}_N:=\Big\{(h,y)\in\mathbb{R}^d\times\mathcal{Y}:|h|>N,|y|>N\Big\}.
\end{equation}
Then the family $\{\mathbf{P}^{\delta,\Delta}\}_{\delta>0}$ is uniformly integrable in the sense that
\begin{equation}\nonumber
\lim\limits_{N\to\infty}\sup\limits_{\delta>0}\mathbb{E}\Bigg\{\int_{\mathcal{U}_{N}\times[0,T]}\big[|h|+|y|\big]\mathbf{P}^{\delta,\Delta}(dhdydt)\Bigg\}=0.
\end{equation}
\end{proposition}

\begin{proof}\noindent$\textbf{Step 1.}~( \mathbf{Tightness}~\mathbf{of}~\{(X^{\delta,h^{\delta}},\mathcal{R}_1^\delta,\mathcal{R}_2^\delta)\}_{\delta>0}$) 
Following from the criterion of tightness (cf.~\cite[Theorem 7.3]{b}), due to the uniform moment estimate (\ref{3.1})  it suffices to prove that for any positive $\theta$, $\eta$, there exists a constant ${\delta}_{0}>0$ such that
\begin{equation}\label{4.2}
\sup\limits_{\delta\in(0,1)}\mathbb{P}\Big(\sup\limits_{t_1,t_2\in[0,T],|t_1-t_2|<\delta_0}|X_{t_1}^{\delta,h^\delta}-X_{t_2}^{\delta,h^\delta}|\geq\theta\Big)\leq\eta.
\end{equation}
Based on Lemma \ref{l3.3}, it is clear that $\{X_{t}^{\delta,h^\delta}\}_{\delta>0}$ satisfies (\ref{4.2}) by using the Kolmogorov's continuity criterion.

For the term $\mathcal{R}_1^\delta(t)$, for any $t_{1}, t_{2} \in [0,T]$ we have
\begin{align}\label{4.3}
	&\mathbb{E}|\mathcal{R}_{1}^{\delta}(t_{1})-\mathcal{R}_{1}^{\delta}(t_{2})|^{4}\nonumber \\
	\leq&\mathbb{E}\left(\int_{t_{2}}^{t_{1}}|b(X_s^{\delta,h^{\delta}},\mathscr{L}_{X_s^{\delta}},Y_s^{\delta,h^{\delta}},\mathscr{L}_{Y_s^{\delta}})|ds\right)^{4} \nonumber\\
	\leq& C\mathbb{E}\Bigg(\int_{t_{2}}^{t_{1}}\Big(1+|X_{s}^{\delta,h^{\delta}}|+\big(\mathbb{E}|X_{s}^{\delta}|^{2}\big)^\frac{1}{2}+|Y_{s}^{\delta,h^{\delta}}|+\big(\mathbb{E}|Y_{s}^{\delta}|^{2}\big)^\frac{1}{2}\Big)ds\Bigg)^4\nonumber\\
	\leq& C|t_{1}-t_{2}|^{2}\Bigg[|t_{1}-t_{2}|^{2}\Big(1+\mathbb{E}\int_{t_2}^{t_1}|X_{s}^{\delta,h^{\delta}}|^{4}ds\Big)
	\nonumber\\
	&+\Big(\int_{t_{2}}^{t_{1}}\mathbb{E}|X_{s}^{\delta}|^{2}ds\Big)^2+\mathbb{E}\Big(\int_{t_{2}}^{t_{1}}|Y_{s}^{\delta,h^{\delta}}|^{2}ds\Big)^2+\Big(\int_{t_{2}}^{t_{1}}\mathbb{E}|Y_{s}^{\delta}|^{2}ds\Big)^2\Bigg]\nonumber\\
	\leq& C_{T}|t_{1}-t_{2}|^{2},
\end{align}
where we used the estimates ({\ref{3.1}}), ({\ref{3.2}}), (\ref{e02}) and (\ref{e03}) in the last step. Note that (\ref{4.3}) implies (\ref{4.2}) with $\mathcal{R}_{1}^{\delta}$ replacing  $X^{\delta,h^{\delta}}$.

Similarly, for the term  $\mathcal{R}_2^\delta(t)$, due to $(\mathbf{A_2})$ and the fact that $h^{\delta}\in\mathcal{A}_{M},$ we have
\begin{align}\label{4.4}
	\mathbb{E}|\mathcal{R}_{2}^{\delta}(t_{1})-\mathcal{R}_{2}^{\delta}(t_{2})|^{4} \leq&\mathbb{E}\Bigg(\int_{t_{2}}^{t_{1}}\|\sigma\big(X_s^{\delta,h^{\delta}},\mathscr{L}_{X_s^{\delta}},Y_s^{\delta,h^{\delta}},\mathscr{L}_{Y_s^{\delta}}\big)\|\cdot|h_{s}^{1,\delta}|ds\Bigg)^{4} \nonumber \\
	\leq& C_{T}\mathbb{E}\Bigg(\int_{t_2}^{t_1}|h_s^{1,\delta}|ds\Bigg)^4 \nonumber \\
	\leq& C_{T}|t_{1}-t_{2}|^{2}\mathbb{E}\Bigg(\int_{0}^{T}|h_{s}^{1,\delta}|^{2}ds\Bigg)^{2} \nonumber \\
	\leq& C_{M,T}|t_{1}-t_{2}|^{2}.
	\end{align}
Therefore, following from (\ref{4.3}) and (\ref{4.4}), we can obtain the tightness of  $\{(\mathcal{R}_1^\delta,\mathcal{R}_2^\delta)\}_{\delta>0}$.

\vspace{1mm}
\noindent$\textbf{Step 2.}~(  \mathbf{Tightness}~\mathbf{of}~\{\mathbf{P}^{\delta,\Delta}\}_{\delta>0})$ We note that the function
\begin{equation}\nonumber
\Psi(\gamma):=\int_{\mathbb{R}^d\times\mathcal{Y}\times[0,T]}\left[|h|^2+|y|^2\right]\gamma(dhdydt), ~~\gamma\in\mathcal{P}(\mathbb{R}^d\times\mathcal{Y}\times[0,T])
\end{equation}
is a tightness function due to the fact that it is nonnegative and that the level set
\begin{equation}\nonumber
\mathcal{H}_k:=\left\{\gamma\in\mathcal{P}(\mathbb{R}^d\times\mathcal{Y}\times[0,T]):\Psi(\gamma)\leq k\right\}
\end{equation}
is relatively compact in $\mathcal{P}(\mathbb{R}^d\times\mathcal{Y}\times[0,T])$, for each $k < \infty$. Indeed,
using Chebyshev's inequality, we have
\begin{equation}\nonumber
\sup\limits_{\gamma\in\mathcal{H}_k}\gamma\Big(\Big\{(h,y,t)\in\mathcal{U}_N\times[0,T]\Big\}\Big)\leq\sup\limits_{\gamma\in\mathcal{H}_k}\frac{\Psi(\gamma)}{N^2}\leq\frac{k}{N^2}.
\end{equation}
Hence, $\mathcal{H}_k$ is tight and thus relatively compact as a subset of $\mathcal{P}(\mathbb{R}^d\times\mathcal{Y}\times[0,T])$.

Since $\Psi$ is a tightness function, by Theorem A.3.17 in \cite{m} the tightness of  $\mathbf{\{P}^{\delta,\Delta}\}_{\delta>0}$ holds if
$$\sup_{\delta>0}\mathbb{E}\big[\Psi(\mathbf{P}^{\delta,\Delta})\big]<\infty.$$ Indeed, by (\ref{e04}) we can get
\begin{align}\label{4.6}
\sup_{\delta>0}\mathbb{E}\big[\Psi(\mathbf{P}^{\delta,\Delta})\big] &=\operatorname*{sup}_{\delta>0}\mathbb{E}\int_{\mathbb{R}^d\times\mathcal{Y}\times[0,T]}\left[|h|^{2}+|y|^{2}\right]\mathbf{P}^{\delta,\Delta}(dhdydt)  \nonumber\\
&=\sup_{\delta>0}\mathbb{E}\int_{0}^{T}\frac{1}{\Delta}\int_{t}^{t+\Delta}\left[|h_{s}^{\delta}|^{2}+|Y_{s}^{\delta,h^{\delta}}|^{2}\right]dsdt \nonumber\\
&\leq C\sup_{\delta>0}\mathbb{E}\int_{0}^{T+\Delta}\Big[|h_{s}^{\delta}|^{2}+|Y_{s}^{\delta,h^{\delta}}|^{2}\Big]ds<\infty.
\end{align}

\noindent$\textbf{Step 3.}~( \mathbf{Uniform}~\mathbf{integrability}~\mathbf{of}~\{\mathbf{P}^{\delta,\Delta}\}_{\delta>0})$ This statement follows from the claim (ii) and the following inequality
\begin{align*}
&\mathbb{E}\int_{\mathcal{U}_N\times[0,T]}\big[|h|+|y|\big]\mathbf{P}^{\delta,\Delta}(dhdydt) \\
\leq&{\frac{C}{N}}\mathbb{E}\int_{\mathbb{R}^d\times\mathcal{Y}\times[0,T]}\big[|h|^{2}+|y|^{2}\big]\mathbf{P}^{\delta,\Delta}(dhdydt).
\end{align*}
The proof is complete.
\end{proof}

\section{Proof of LDP}
Building upon Proposition \ref{p4.1} and Prokhorov theorem, for any subsequence of $\{(X^{\delta,h^{\delta}},\mathcal{R}_1^\delta,\mathcal{R}_2^\delta,\mathbf{P}^{\delta,\Delta})\}_{\delta>0}$  there exists a subsubsequence still denoted by $(X^{\delta,h^{\delta}},\mathcal{R}_1^\delta,\mathcal{R}_2^\delta,\mathbf{P}^{\delta,\Delta})$ such that
$$(X^{\delta,h^\delta},\mathcal{R}_1^\delta,\mathcal{R}_2^\delta,\mathbf{P}^{\delta,\Delta})\Rightarrow(X,\mathcal{R}_1,\mathcal{R}_2,\mathbf{P})~~\text{in}~\mathcal{Z}_T,\quad\text{as}~\delta\to0,$$
where we denote by ``$\Rightarrow$" the weak convergence of random variables. Then applying the Skorokhod representation theorem, one can construct another probability space along with random variables denoted
by $((\tilde{\Omega},\tilde{\mathcal{F}},\tilde{\mathbb{P}}),\tilde{X}^{\delta,h^{\delta}},\tilde{\mathcal{R}}_1^\delta,\tilde{\mathcal{R}}_2^\delta,\tilde{\mathbf{P}}^{\delta,\Delta})$ such that
\begin{align}
&(i)~~(\tilde{X}^{\delta,h^\delta},\tilde{\mathcal{R}}_1^\delta,\tilde{\mathcal{R}}_2^\delta,\tilde{\mathbf{P}}^{\delta,\Delta})\to(\tilde{X},\tilde{\mathcal{R}}_1,\tilde{\mathcal{R}}_2,\tilde{\mathbf{P}})~~\text{in}~\mathcal{Z}_T~~\tilde{\mathbb{P}}\text{-a.s.},~\text{as}~\delta\to0,\label{4.7}
\\
&(ii)~~\mathscr{L}_{(X^{\delta,h^\delta},\mathcal{R}_1^\delta,\mathcal{R}_2^\delta,\mathbf{P}^{\delta,\Delta})}|_{\mathbb{P}}=\mathscr{L}_{(\tilde{X}^{\delta,h^\delta},\tilde{\mathcal{R}}_1^\delta,\tilde{\mathcal{R}}_2^\delta,\tilde{\mathbf{P}}^{\delta,\Delta})}|_{\tilde{\mathbb{P}}},~
\mathscr{L}_{(X,\mathcal{R}_1,\mathcal{R}_2,\mathbf{P})}|_{\mathbb{P}}=\mathscr{L}_{(\tilde{X},\tilde{\mathcal{R}}_1,\tilde{\mathcal{R}}_2,\tilde{\mathbf{P}})}|_{\tilde{\mathbb{P}}},\label{4.16}
\end{align}
where we denote by $\mathscr{L}_{X}|_{\mathbb{P}}$ the law of $X$ under the probability measure $\mathbb{P}$.   Hence, our next objective is to demonstrate that the accumulation point $(X,\mathbf{P})$ is a viable pair w.r.t $(\Phi,\nu,x,\bar{X})$ in the sense of Definition \ref{d2.4}, i.e. $(X,\mathbf{P})\in\mathcal{V}_{(\Phi,\nu,x,\bar{X})}.$

\subsection{Existence of viable pair}
First, by using Fatou's lemma and (\ref{4.6}), it follows that
\begin{align*}
&\mathbb{E}\int_{\mathbb{R}^d\times\mathcal{Y}\times[0,T]}\left[|h|^{2}+|y|^{2}\right]\mathbf{P}(dhdydt) \\
\leq&\liminf_{\delta\to0}\mathbb{E}\int_{\mathbb{R}^d\times\mathcal{Y}\times[0,T]}\left[|h|^{2}+|y|^{2}\right]\mathbf{P}^{\delta,\Delta}(dhdydt) \\
\leq&\sup_{\delta>0}\mathbb{E}\big[\Psi(\mathbf{P}^{\delta,\Delta})\big]<\infty,
\end{align*}
which yields that
$$
\int_{\mathbb{R}^d \times\mathcal{Y}\times[0,T]}\big[|h|^2+|y|^2\big]\mathbf{P}(dhdydt)<\infty,\mathbb{P}\text{-a.s..}
$$
Hence, the claim (i) in Definition \ref{d2.4} is satisfied. It remains to prove that the claims (\ref{2.16})-(\ref{2.17}) hold for $(X,\mathbf{P})$.

In the following, we first prove (\ref{2.18}), which will be used in proving (\ref{2.16}).
\vspace{1mm}

\noindent$\mathbf{Proof~of~(\ref{2.18}).}$ Recall the fact that $\mathbf{P}^{\delta,\Delta}(\mathbb{R}^d\times\mathcal{Y}\times[0,t])=t,$ and $\mathbf{P}(\mathbb{R}^d\times\mathcal{Y}\times{t})=0.$ Utilizing the continuity of $t \mapsto\mathbf{P}(\mathbb{R}^d\times\mathcal{Y}\times[0,t])$
to deal with null sets, it follows that (\ref{2.18}) holds.

From now on, we prove (\ref{2.16})-(\ref{2.17}).
\vspace{1mm}

\noindent$\mathbf{Proof~of~(\ref{2.16}).}$ Recall the equality (\ref{4.1})
\begin{align}\label{4.0}
	X_t^{\delta,h^{\delta}}=&~x+\sum_{i=1}^3\mathcal{R}_i^\delta(t),
\end{align}
where
\begin{align*}
&\mathcal{R}_1^\delta(t)=\int_{0}^{t}b(X_s^{\delta,h^{\delta}},\mathscr{L}_{X_s^{\delta}},Y_s^{\delta,h^{\delta}},\mathscr{L}_{Y_s^{\delta}})ds,
\nonumber\\
&\mathcal{R}_2^\delta(t)=\int_{0}^{t}\sigma(X_s^{\delta,h^{\delta}},\mathscr{L}_{X_s^{\delta}},Y_s^{\delta,h^{\delta}},\mathscr{L}_{Y_s^{\delta}})h_s^{1,\delta}ds, \nonumber\\
	&\mathcal{R}_3^\delta(t)=\sqrt{\delta}\int_{0}^{t}\sigma(X_s^{\delta,h^{\delta}},\mathscr{L}_{X_s^{\delta}},Y_s^{\delta,h^{\delta}},\mathscr{L}_{Y_s^{\delta}})dW_s^1.
\end{align*}

Our next objective is to show the convergence of the terms $\mathcal{R}_{i}^{\delta}(t),i=1,2,3.$ Particularly, in view of the condition (\ref{rf1}), we note that the term $\mathcal{R}_{3}^{\delta}(t)$ vanishes in probability in $C([0,T];\mathbb{R}^{n})$, as $\delta \to 0$. Therefore, it is sufficient to demonstrate the convergence of the remaining terms. To accomplish this, we will divide the proof into the following Lemmas $\ref{l4.1}$-$\ref{l4.2}$.

\begin{lemma}\label{l4.1}
Let $t\in[0,T]$.	The following limit is valid in distribution:
	\begin{equation*}\label{4.8}
		\mathcal{R}_{2}^{\delta}(t)\xrightarrow{\delta\to0}\int_{\mathbb{R}^d\times\mathcal{Y}\times[0,t]}\sigma(X_{s},\mathscr{L}_{\bar{X}_{s}},y,\nu)P_1h\mathbf{P}(dhdyds).
	\end{equation*}
\end{lemma}
\begin{proof}
First,  we have the following composition
\begin{align*}
&\int_{0}^{t}\sigma(X_{s}^{\delta,h^{s}},\mathscr{L}_{X_{s}^{\delta}},Y_{s}^{\delta,h^{\delta}},\mathscr{L}_{Y_{s}^{\delta}})h_{s}^{1,\delta}ds \\
=&\bigg(\int_{0}^{t}\sigma(X_{s}^{\delta,h^{\delta}},\mathscr{L}_{X_{s}^{\delta}},Y_{s}^{\delta,h^{\delta}},\mathscr{L}_{Y_{s}^{\delta}})h_{s}^{1,\delta}ds -\int_{0}^{t}\sigma(X_{s}^{\delta,h^{\delta}},\mathscr{L}_{X_{s}^{\delta}},Y_{s}^{\delta,h^{\delta}},\nu)h_{s}^{1,\delta}ds \bigg) \\
&+\bigg(\int_{0}^{t}\sigma(X_{s}^{\delta,h^{\delta}},\mathscr{L}_{X_{s}^{\delta}},Y_{s}^{\delta,h^{\delta}},\nu)h_{s}^{1,\delta}ds\\
&-\int_{\mathbb{R}^d\times\mathcal{Y}\times[0,t]}\sigma(X_s^{\delta,h^{\delta}},\mathscr{L}_{X_s^\delta},y,\nu)P_1h\mathbf{P}^{\delta,\Delta}(dhdyds)\bigg) \\
&+\int_{\mathbb{R}^d\times\mathcal{Y}\times[0,t]}\sigma(X_{s}^{\delta,h^{\delta}},\mathscr{L}_{X_{s}^{\delta}},y,\nu)P_1h\mathbf{P}^{\delta,\Delta}(dhdyds) \\
=:&\sum_{i=1}^{3}\mathcal{O}_{i}^{\delta}(t).
\end{align*}

\noindent$\textbf{Step 1.}$ In this part, we deal with terms $\mathcal{O}_i^\delta(t), i=1,2$.  By $(\mathbf{A_1})$, we observe that
\begin{align}\label{zz5.9}
&\mathbb{E}\Big[\sup_{t\in[0,T]}|\mathcal{O}_1^\delta(t)|^2\Big]
\nonumber\\
\leq&\mathbb{E}\Bigg(\int_{0}^{T}\|\sigma(X_{t}^{\delta,h^{\delta}},\mathscr{L}_{X_{t}^{\delta}},Y_{t}^{\delta,h^{\delta}},\mathscr{L}_{Y_{t}^{\delta}})-\sigma(X_{t}^{\delta,h^{\delta}},\mathscr{L}_{X_{t}^{\delta}},Y_{t}^{\delta,h^{\delta}},\nu)\|^{2}dt\cdot\int_{0}^{T}|h_{t}^{1,\delta}|^2dt\Bigg) \nonumber\\
\leq& C_{M,T}\int_{0}^{T}\mathbb{W}_2(\mathscr{L}_{Y_{t}^{\delta}},\nu)^2dt\nonumber\\
\leq& C_{M,T}\int_0^T\mathbb{W}_{2}(\mathscr{L}_{Y_\frac{t}{\varepsilon}^{y}},\nu)^2dt \to 0,~~\text{as}~\delta \to 0,
\end{align}
where ${Y_{t}^{y}}$ is the solution of (\ref{n2}), and the last step follows from \cite[Lemma 4.3]{vvv}.

Secondly, we estimate the term $\mathcal{O}_2^\delta(t)$. By (\ref{2.13}), we observe that
\begin{align*}
&\int_{\mathbb{R}^d\times\mathcal{Y}\times[0,t]}\sigma(X_s^{\delta,h^\delta},\mathscr{L}_{X_s^\delta},y,\nu)P_1h\mathbf{P}^{\delta,\Delta}(dhdyds) \\
=& \left(\int_{0}^{t}\frac{1}{\Delta}\int_{s}^{s+\Delta}\sigma(X_{s}^{\delta,h^{\delta}},\mathscr{L}_{X_{s}^{\delta}},Y_{r}^{\delta,h^{\delta}},\nu)h_{r}^{1,\delta}drds\right.  \\
&-\int_0^t\frac1\Delta\int_s^{s+\Delta}\sigma(X_r^{\delta,h^\delta},\mathscr{L}_{X_r^\delta},Y_r^{\delta,h^\delta},\nu)h_r^{1,\delta}drds\Bigg) \\
&+\int_0^t\frac1\Delta\int_s^{s+\Delta}\sigma(X_r^{\delta,h^\delta},\mathscr{L}_{X_r^\delta},Y_r^{\delta,h^\delta},\nu)h_r^{1,\delta}drds \\
=:&I_1^\delta(t)+I_2^\delta(t).
\end{align*}
As for $I_1^\delta(t)$, due to the Lipschitz continuity of $\sigma$ and (\ref{3.12}), it follows that
\begin{align}\label{4.9}
&\mathbb{E}\Big[\sup_{t\in[0,T]}|I_1^\delta(t)|^2\Big]
\nonumber\\
 \leq&\frac{1}{\Delta^{2}}\mathbb{E}\bigg(\int_{0}^{T}\int_{s}^{s+\Delta}\big(|X_{s}^{\delta,h^{\delta}}-X_{r}^{\delta,h^{\delta}}|^{2}+\mathbb{E}|X_{s}^{\delta}-X_{r}^{\delta}|^{2}\big)drds\cdot\int_{0}^{T}\int_{s}^{s+\Delta}|h_{r}^{1,\delta}|^{2}drds\bigg) \nonumber \\
\leq&\frac{C_{M,T}}{\Delta}\int_{0}^{T}\int_{s}^{s+\Delta}\left(\mathbb{E}|X_{s}^{\delta,h^{\delta}}-X_{r}^{\delta,h^{\delta}}|^{2}+\mathbb{E}|X_{s}^{\delta}-X_{r}^{\delta}|^{2}\right)drds \nonumber \nonumber \\
\leq&\frac{C_{M,T}}\Delta\int_0^T\int_s^{s+\Delta}(r-s)drds \nonumber\\
\leq& C_{M,T}\Delta\to0,~~\mathrm{as}~\delta\to0,
\end{align}
where the second step is due to the fact that for any $t \in [0, T],$
\begin{align*}
\int_{0}^{t}\int_{s}^{s+\Delta}|h_{r}^{1,\delta}|^{2}drds =&\int_{0}^{\Delta}\int_{0}^{r}|h_{r}^{1,\delta}|^{2}dsdr+\int_{\Delta}^{t}\int_{r-\Delta}^{r}|h_{r}^{1,\delta}|^{2}dsdr  \\
&+\int_t^{t+\Delta}\int_{r-\Delta}^t|h_r^{1,\delta}|^2dsdr \\
=&\int_{0}^{\Delta}|h_{r}^{1,\delta}|^{2}rdr+\Delta\int_{\Delta}^{t}|h_{r}^{1,\delta}|^{2}dr+\int_{t}^{t+\Delta}(t-r+\Delta)|h_{r}^{1,\delta}|^{2}dr \\
\leq&3\Delta\int_0^T|h_r^{1,\delta}|^2dr.
\end{align*}
On the other hand, the term $I_2^\delta(t)$ has the composition
\begin{align*}
I_{2}^{\delta}(t) =&\int_{0}^{\Delta}\frac{1}{\Delta}\int_{0}^{r}\sigma(X_{r}^{\delta,h^{\delta}},\mathscr{L}_{X_{r}^{\delta}},Y_{r}^{\delta,h^{\delta}},\nu)h_{r}^{1,\delta}dsdr  \\
&+\int_{\Delta}^{t}\frac{1}{\Delta}\int_{r-\Delta}^{r}\sigma(X_{r}^{\delta,h^{\delta}},\mathscr{L}_{X_{r}^{\delta}},Y_{r}^{\delta,h^{\delta}},\nu)h_{r}^{1,\delta}dsdr \\
&+\int_t^{t+\Delta}\frac{1}{\Delta}\int_{r-\Delta}^t\sigma(X_r^{\delta,h^\delta},\mathscr{L}_{X_r^\delta},Y_r^{\delta,h^\delta},\nu)h_r^{1,\delta}dsdr \\
=:&\sum_{i=1}^{3}I_{2i}^{\delta}(t).
\end{align*}
Due to the boundedness of $\sigma$$\sigma^*$ and $h_{r}^{1,\delta} \in \mathcal{A}_M$, it is easy to see that
\begin{equation}\label{4.10}
\mathbb{E}\Big[\sup_{t\in[0,T]}|I_{21}^{\delta}(t)|^2\Big]+\mathbb{E}\Big[\sup_{t\in[0,T]}|I_{23}^{\delta}(t)|^2\Big]\to0,~~\text{as}~\delta\to0.
\end{equation}
Thus, for the term $\mathcal{O}_2^\delta(t)$, it follows that
\begin{align*}
&\mathbb{E}\Big[\sup_{t\in[0,T]}|\mathcal{O}_{2}^{\delta}(t)|^{2}\Big]
\\
\leq& C\mathbb{E}\Big|\int_{0}^{\Delta}\sigma(X_{r}^{\delta,h^{\delta}},\mathscr{L}_{X_{r}^{\delta}},Y_{r}^{\delta,h^{\delta}},\nu)h_{r}^{1,\delta}dr\Big|^{2}  \\
&+C\mathbb{E}\Bigg[\sup_{t\in[0,T]}\left|\int_{\Delta}^{t}\sigma(X_{r}^{\delta,h},\mathscr{L}_{X_{r}^{\delta}},Y_{r}^{\delta,h^{\delta}},\nu)h_{r}^{1,\delta}dr-I_{22}^{\delta}(t)\right|^{2}\Bigg] \\
&+C\Bigg(\mathbb{E}\Big[\sup_{t\in[0,T]}|I_{1}^{\delta}(t)|^{2}\Big]+\mathbb{E}\Big[\sup_{t\in[0,T]}|I_{21}^{\delta}(t)|^{2}\Big]+\mathbb{E}\Big[\sup_{t\in[0,T]}|I_{23}^{\delta}(t)|^{2}\Big]\Bigg).
\end{align*}
Then by the boundedness of $\sigma$$\sigma^*$ and $h_{r}^{1,\delta} \in \mathcal{A}_M$, (\ref{4.9}), (\ref{4.10}) and the definition of $I_{22}^{\delta}(t)$, it leads to
\begin{equation}\label{4.11}
\mathbb{E}\Big[\sup_{t\in[0,T]}|\mathcal{O}_2^\delta(t)|^2\Big]\to0,\quad\text{as}~\delta\to0.
\end{equation}

\noindent$\textbf{Step 2.}$ Denote
\begin{equation*}
\mathcal{O}_3(t):=\int_{\mathbb{R}^d\times\mathcal{Y}\times[0,t]}\sigma(X_s,\mathscr{L}_{\bar{X}_s},y,\nu)P_1h\mathbf{P}(dhdyds).
\end{equation*}
We shall show that the following convergence holds 
	\begin{equation}\label{4.19}
		\mathcal{O}_3^\delta(t) \xrightarrow{\text{Law}} \mathcal{O}_3 (t).
	\end{equation}
Notice that for any $t\in[0,T]$ and bounded Lipschitz function $f$,  we have
\begin{align}\label{Skoro}
\mathbb{E}\big[f(\mathcal{O}_3^\delta(t))-f(\mathcal{O}_3(t))\big]=\tilde{\mathbb{E}}\big[f(\tilde{\mathcal{O}}_3^{\delta}(t))-f(\tilde{\mathcal{O}}_3(t))\big]\leq \tilde{\mathbb{E}}\big|\tilde{\mathcal{O}}_3^{\delta}(t)-\tilde{\mathcal{O}}_3(t)\big|,
\end{align}
where 
\begin{align*}
&\tilde{\mathcal{O}}_3^{\delta}(t):= \int_{\mathbb{R}^d\times\mathcal{Y}\times[0,t]}\sigma(\tilde{X}_{s}^{\delta,h^{\delta}},\mathscr{L}_{X_{s}^{\delta}},y,\nu)P_1h\tilde{\mathbf{P}}^{\delta,\Delta}(dhdyds), \nonumber\\
&\tilde{\mathcal{O}}_3(t):=\int_{\mathbb{R}^d\times\mathcal{Y}\times[0,t]}\sigma(\tilde{X}_s,\mathscr{L}_{\bar{X}_s},y,\nu)P_1h\tilde{\mathbf{P}}(dhdyds),
\end{align*}
and we used (\ref{4.16}) in the first step.

To show (\ref{4.19}), it suffices to prove that the right-hand side of (\ref{Skoro}) vanishes as $\delta\to 0$.
 As for the term $\tilde{\mathcal{O}}_3^{\delta}(t)-\tilde{\mathcal{O}}_3(t)$, we have the following composition
 \begin{align*}
 \tilde{\mathcal{O}}_3^{\delta}(t)-\tilde{\mathcal{O}}_3(t)=&\Bigg(\int_{\mathbb{R}^d\times\mathcal{Y}\times[0,t]}\sigma(\tilde{X}_{s}^{\delta,h^{\delta}},\mathscr{L}_{X_{s}^{\delta}},y,\nu)P_1h\tilde{\mathbf{P}}^{\delta,\Delta}(dhdyds) \\
 &-\int_{\mathbb{R}^d\times\mathcal{Y}\times[0,t]}\sigma(\tilde{X}_s,\mathscr{L}_{\bar{X}_s},y,\nu)P_1h\tilde{\mathbf{P}}^{\delta,\Delta}(dhdyds)\bigg) \\
 &+\Bigg(\int_{\mathbb{R}^d\times\mathcal{Y}\times[0,t]}\sigma(\tilde{X}_{s},\mathscr{L}_{\bar{X}_{s}},y,\nu)P_1h\tilde{\mathbf{P}}^{\delta,\Delta}(dhdyds) \\
 &-\int_{\mathbb{R}^d\times\mathcal{Y}\times[0,t]}\sigma(\tilde{X}_s,\mathscr{L}_{\bar{X}_s},y,\nu)P_1h\tilde{\mathbf{P}}(dhdyds)\bigg) \\
 =:&\sum_{i=1}^{2}\tilde{\mathcal{J}}_{i}^{\delta}(t).
 \end{align*}
Concerning the term $\tilde{\mathcal{J}}_{1}^{\delta}(t)$, we have the following estimates
\begin{align*}
&\tilde{\mathbb{E}}\Big[\sup_{t\in[0,T]}|\tilde{\mathcal{J}}_{1}^{\delta}(t)|\Big] \\
\leq&\tilde{\mathbb{E}}\Bigg[\operatorname*{sup}_{t\in[0,T]}\left|\int_{\mathbb{R}^d\times\mathcal{Y}\times[0,t]}\left(\sigma(\tilde{X}_s^{\delta,h^{\delta}},\mathscr{L}_{X_s^\delta},y,\nu)-\sigma\big(\tilde{X}_s,\mathscr{L}_{\bar{X}_{s}},y,\nu\big)\right)P_1h\tilde{\mathbf{P}}^{\delta,\Delta}(dhdyds)\right| \Bigg] \\
\leq&\tilde{\mathbb{E}}\Bigg(\int_{\mathbb{R}^d\times\mathcal{Y}\times[0,T]}\|\sigma(\tilde{X}_s^{\delta,h^{\delta}},\mathscr{L}_{X_s^\delta},y,\nu)-\sigma\big(\tilde{X}_s,\mathscr{L}_{\bar{X}_{s}},y,\nu\big)\|^{2}\tilde{\mathbf{P}}^{\delta,\Delta}(dhdyds) \\
&\cdot\int_{\mathbb{R}^d\times\mathcal{Y}\times[0,T]}|h|^2\tilde{\mathbf{P}}^{\delta,\Delta}(dhdyds)\Bigg)^{\frac{1}{2}} \\
\leq& C_{T}\bigg\{\tilde{\mathbb{E}}\Big[\sup_{t\in[0,T]}|\tilde{X}_{t}^{\delta,h^{\delta}}-\tilde{X}_{t}|^{2}\Big]+\sup_{t\in[0,T]}\mathbb{E}|X_{t}^{\delta}-\bar{X}_{t}|^{2}\bigg\}^{\frac{1}{2}}.
\end{align*}
Therefore, by the estimate (\ref{3.1}) and the convergence (\ref{z3.1}) and (\ref{4.7}) we can deduce
\begin{equation}\label{4.12}
\tilde{\mathbb{E}}\Big[\sup\limits_{t\in[0,T]}|\tilde{\mathcal{J}}_{1}^{\delta}(t)|\Big]\to0,\quad\mathrm{as}~\delta\to0.
\end{equation}

Now we turn to study the convergence of $\tilde{\mathcal{J}}_{2}^{\delta}(t)$. It's noteworthy that $\{\tilde{\mathbf{P}}^{\delta,\Delta}\}_{\delta>0}$ is $L^1$-uniformly integrable from Proposition \ref{p4.1} and (\ref{4.16}). Then, by (\ref{4.7}) it is clear that
\begin{equation*}
	\mathbb{W}_1(\tilde{\mathbf{P}}^{\delta,\Delta},\tilde{\mathbf{P}})\to 0\quad\tilde{\mathbb{P}}\text{-a.s.},~~\text{as}~\delta\to0.
\end{equation*}
Therefore, leveraging the assumption $(\mathbf{A_2})$, we deduce the following convergence
\begin{align}\label{4.13}
&\int_{\mathbb{R}^d\times\mathcal{Y}\times[0,T]}   \sigma(\tilde{X}_{s},\mathscr{L}_{\bar{X}_{s}},y,\nu)P_1h\tilde{\mathbf{P}}^{\delta,\Delta}(dhdyds)\nonumber\\
&\to\int_{\mathbb{R}^d\times\mathcal{Y}\times[0,T]}   \sigma(\tilde{X}_{s},\mathscr{L}_{\bar{X}_{s}},y,\nu)P_1h\tilde{\mathbf{P}}(dhdyds)\quad\tilde{\mathbb{P}}\text{-a.s.},~~\mathrm{as}~\delta\to0.
\end{align}
Combining (\ref{Skoro}), (\ref{4.12}) and  (\ref{4.13}), we can deduce that the convergence (\ref{4.19}) holds.

Following from  (\ref{zz5.9}), (\ref{4.11}) and (\ref{4.19}), we conclude  that Lemma \ref{l4.1} holds. We complete the proof.
\end{proof}

\begin{lemma}\label{l4.2}
Let $t\in[0,T]$.	The following limit is valid in distribution:
	\begin{align}\label{4.14}
		\mathcal{R}_{1}^{\delta}(t) \xrightarrow{\delta\to0} \int_{\mathbb{R}^d\times\mathcal{Y}\times[0,t]}\bar{b}(X_{s},\mathscr{L}_{\bar{X}_{s}})\mathbf{P}(dhdyds).
	\end{align}
\end{lemma}
\begin{proof}
Note that by (\ref{2.18}) it is clear that
\begin{align*}
\int_{\mathbb{R}^d\times\mathcal{Y}\times[0,t]}\bar{b}(X_{s},\mathscr{L}_{\bar{X}_{s}})\mathbf{P}(dhdyds)=\int_{0}^{t}\bar{b}(X_{s},\mathscr{L}_{\bar{X}_{s}})ds.
\end{align*}
Thus, it is sufficient to show that the following limit holds 
\begin{equation}\label{4.15}
	\mathcal{R}_{1}^{\delta}(t) \xrightarrow{\text{Law}}  \int_{0}^{t}\bar{b}(X_{s},\mathscr{L}_{\bar{X}_{s}})ds.
\end{equation}

The proof of (\ref{4.15}) is divided into the following two steps.

\vspace{1mm}
\noindent$\mathbf{Step~1.}$ Note that
\begin{align}\label{e05}
\mathcal{R}_{1}^{\delta}(t)=&\int_{0}^{t} b(X_{s}^{\delta,h^{\delta}},\mathscr{L}_{X_{s}^{\delta}},Y_{s}^{\delta,h^{\delta}},\mathscr{L}_{Y_{s}^{\delta}}) -b(X_{s(\Delta)}^{\delta,h^{\delta}},\mathscr{L}_{X_{s(\Delta)}^{\delta}},Y_{s}^{\delta},\nu)ds \nonumber\\
&+\int_{0}^{t} b(X_{s(\Delta)}^{\delta,h^{\delta}},\mathscr{L}_{X_{s(\Delta)}^{\delta}},Y_{s}^{\delta},\nu)-\bar{b}(X_{s(\Delta)}^{\delta,h^{\delta}},\mathscr{L}_{X_{s(\Delta)}^{\delta}})ds \nonumber\\
&+\int_{0}^{t} \bar{b}(X_{s(\Delta)}^{\delta,h^{\delta}},\mathscr{L}_{X_{s(\Delta)}^{\delta}})-\bar{b}(X_{s}^{\delta,h^{\delta}},\mathscr{L}_{X_{s}^{\delta}}) ds\nonumber\\
&+\int_{0}^{t}
\bar{b}(X_{s}^{\delta,h^{\delta}},\mathscr{L}_{X_{s}^{\delta}})ds\nonumber\\
=:& {I}_1(t) +{I}_2(t) +{I}_3(t) +{I}_4(t),
\end{align}
where $s(\Delta):=[\frac{s}{\Delta}]\Delta$ and $[s]$ denotes the integer part of $s$.

By the Lipschitz continuity of $b$ and $\bar{b}$ we have
\begin{align}\label{zz4.17}
&\mathbb{E}\Big[\sup_{t\in[0,T]}|{I}_1(t)+{I}_3(t)|^2\Big]
\nonumber\\
\leq& C\mathbb{E}\int_{0}^{T}|X_{t}^{\delta,h^{\delta}}-X_{t(\Delta)}^{\delta,h^{\delta}}|^{2}dt+C\int_{0}^{T}\mathbb{W}_{2}(\mathscr{L}_{X_{t}^{\delta}},\mathscr{L}_{X_{t(\Delta)}^{\delta}})^{2}dt \nonumber\\
&+C\mathbb{E}\int_0^T|Y_{t}^{\delta,h^{\delta}}-Y_{t}^{\delta}|^2dt+C\int_0^T\mathbb{W}_{2}(\mathscr{L}_{Y_{t}^{\delta}},\nu)^2dt \nonumber \\
\leq& C_{M,T}\big(1+|x|^2+|y|^2\big)\Big(\Delta+\frac{\varepsilon}{\delta}+\varepsilon^2\Big),
\end{align}
which the second step follows from the estimate (\ref{3.12}) and Lemmas 2.3, 3.6 and  4.3 in \cite{vvv}.

Similar to the proof of (\ref{4.19}), in view of  (\ref{z3.1}), (\ref{3.1}) and (\ref{4.7}) we also have
\begin{align*}
&\tilde{\mathbb{E}}\Big|\int_{0}^{t}\big(
\bar{b}(\tilde{X}_{s}^{\delta,h^{\delta}},\mathscr{L}_{X_{s}^{\delta}})- \bar{b}(\tilde{X}_{s},\mathscr{L}_{\bar{X}_{s}})\big)ds\Big|^2 \\
\leq& C\tilde{\mathbb{E}}\int_{0}^{t}|\tilde{X}_{s}^{\delta,h^{\delta}}-\tilde{X}_{s}|^{2}ds+C\int_{0}^{t}\mathbb{W}_{2}(\mathscr{L}_{X_{s}^{\delta}},\mathscr{L}_{\bar{X}_{s}})^{2}ds \nonumber\\
\to&~ 0, ~~\text{as} ~\delta \to 0,
\end{align*}
which implies 
the following limit holds 
\begin{equation}\label{e06}
	{I}_4(t) \xrightarrow{\text{Law}}\int_{0}^{t}\bar{b}(X_{s},\mathscr{L}_{\bar{X}_{s}})ds.
\end{equation}

Combining (\ref{e05})-(\ref{e06}), once we can prove
\begin{align}\label{e07}
\mathbb{E}\Big[\sup_{t\in[0,T]}|{I}_2(t)|^2\Big]\to 0&, ~~\text{as} ~\delta \to 0.
\end{align}
then (\ref{4.15}) holds.

\vspace{1mm}
\noindent$\mathbf{Step~2.}$ In this part, we prove (\ref{e07}). We note that
\begin{align*}
|{I}_{2}(t)|^{2} =&\Bigg|\sum_{k=0}^{[t/\Delta]-1}\int_{k\Delta}^{(k+1)\Delta}b(X_{s(\Delta)}^{\delta,h^{\delta}},\mathscr{L}_{X_{s(\Delta)}^{\delta}},Y_{s}^{\delta},\nu)-\bar{b}(X_{s(\Delta)}^{\delta,h^{\delta}},\mathscr{L}_{X_{s(\Delta)}^{\delta}})ds \nonumber \\
&+\int_{t(\Delta)}^{t}b(X_{s(\Delta)}^{\delta,h^{\delta}},\mathscr{L}_{X_{s(\Delta)}^{\delta}},Y_{s}^{\delta},\nu)-\bar{b}(X_{s(\Delta)}^{\delta,h^{\delta}},\mathscr{L}_{X_{s(\Delta)}^{\delta}})ds \Bigg|^2 \nonumber \\
\leq&\frac{C_{T}}{\Delta}\sum_{k=0}^{[t/\Delta]-1}\left|\int_{k\Delta}^{(k+1)\Delta}b(X_{s(\Delta)}^{\delta,h^{\delta}},\mathscr{L}_{X_{s(\Delta)}^{\delta}},Y_{s}^{\delta},\nu)-\bar{b}(X_{s(\Delta)}^{\delta,h^{\delta}},\mathscr{L}_{X_{s(\Delta)}^{\delta}})ds\right|^{2} \nonumber\\
&+2\left|\int_{t(\Delta)}^{t}b(X_{s(\Delta)}^{\delta,h^{\delta}},\mathscr{L}_{X_{s(\Delta)}^{\delta}},Y_{s}^{\delta},\nu)-\bar{b}(X_{s(\Delta)}^{\delta,h^{\delta}},\mathscr{L}_{X_{s(\Delta)}^{\delta}})ds\right|^{2} \nonumber\\
=:&{I}_{21}(t)+{I}_{22}(t).
\end{align*}
In view of term ${I}_{22}(t)$, it follows that
\begin{align*}
\mathbb{E}\Big[\sup_{t\in[0,T]}{I}_{22}(t)\Big] \leq& C\Delta\mathbb{E}\Bigg[\sup_{t\in[0,T]}\int_{t(\Delta)}^{t}\left|b(X_{s(\Delta)}^{\delta,h^{\delta}},\mathscr{L}_{X_{s(\Delta)}^{\delta}},Y_{s}^{\delta},\nu)\right|^{2}+\left|\bar{b}(X_{s(\Delta)}^{\delta,h^{\delta}},\mathscr{L}_{X_{s(\Delta)}^{\delta}})\right|^{2}ds \Bigg] \nonumber
\\
\leq& C\Delta\mathbb{E}\Bigg[\sup_{t\in[0,T]}\int_{t(\Delta)}^{t}\Big(1+|X_{s(\Delta)}^{\delta,h^{\delta}}|^{2}+\mathscr{L}_{X_{s(\Delta)}^{\delta}}(|\cdot|^{2})+|Y_{s}^{\delta}|^{2}+\nu(|\cdot|^{2})\Big)ds\Bigg] \nonumber\\
\leq &C_{M,T}\Delta(1+|x|^2+|y|^2).
\end{align*}

Now we focus on ${I}_{21} (t)$. For simplicity, we denote $b^{\nu} (x,\mu,y) := b(x,\mu,y,\nu)$. Then we rewrite ${I}_{21} (t)$ as
\begin{equation}
{I}_{21}(t)= \frac{C_{T}}{\Delta}\sum_{k=0}^{[t/\Delta]-1}\left|\int_{k\Delta}^{(k+1)\Delta}b^{\nu} (X_{s(\Delta)}^{\delta,h^{\delta}},\mathscr{L}_{X_{s(\Delta)}^{\delta}},Y_{s}^{\delta})-\bar{b}(X_{s(\Delta)}^{\delta,h^{\delta}},\mathscr{L}_{X_{s(\Delta)}^{\delta}})ds\right|^{2}. \nonumber
\end{equation}
Then it can be estimated by
\begin{align}\label{z4.21}
&\mathbb{E}\Big[\sup_{t\in[0,T]}{I}_{21}(t)\Big]
\nonumber\\
\leq& \frac{C_{T}}{\Delta}\mathbb{E}\sum_{k=0}^{[T/\Delta]-1}\left|\int_{k\Delta}^{(k+1)\Delta}b^{\nu} (X_{k\Delta}^{\delta,h^{\delta}},\mathscr{L}_{X_{k\Delta}^{\delta}},Y_{s}^{\delta})-\bar{b}(X_{k\Delta}^{\delta,h^{\delta}},\mathscr{L}_{X_{k\Delta}^{\delta}})ds\right|^{2} \nonumber\\
\leq&\frac{C_T}{\Delta^2}\max_{0\leq k\leq[T/\Delta]-1}\mathbb{E}\left|\int_{k\Delta}^{(k+1)\Delta}b^{\nu} (X_{k\Delta}^{\delta,h^{\delta}},\mathscr{L}_{X_{k\Delta}^{\delta}},Y_{s}^{\delta})-\bar{b}(X_{k\Delta}^{\delta,h^{\delta}},\mathscr{L}_{X_{k\Delta}^{\delta}})ds \right|^2 \nonumber\\
\leq &\frac{C_T\varepsilon^2}{\Delta^2}\max\limits_{0\leq k\leq[T/\Delta]-1}\mathbb{E}\left|\int_0^{\frac{\Delta}{\varepsilon}}b^{\nu} (X_{k\Delta}^{\delta,h^{\delta}},\mathscr{L}_{X_{k\Delta}^{\delta}},Y_{s\varepsilon+k\Delta}^\delta)-\bar{b}(X_{k\Delta}^{\delta,h^{\delta}},\mathscr{L}_{X_{k\Delta}^{\delta}})ds \right|^2 \nonumber\\
\leq&\frac{C_T\varepsilon^2}{\Delta^2}\max_{0\leq k\leq[T/\Delta]-1}\left[\int_0^{\frac{\Delta}{\varepsilon}}\int_r^{\frac{\Delta}{\varepsilon}}\Psi_k(s,r)dsdr\right],
\end{align}
where for any $0\leq r\leq s\leq\frac{\Delta}{\varepsilon},$
\begin{align*}
\Psi_{k}(s,r):=\mathbb{E}\Big[\langle b^{\nu} (X_{k\Delta}^{\delta,h^{\delta}},\mathscr{L}_{X_{k\Delta}^{\delta}},Y_{s\varepsilon+k\Delta}^\delta)-\bar{b}(X_{k\Delta}^{\delta,h^{\delta}},\mathscr{L}_{X_{k\Delta}^{\delta}}),\\
b^{\nu} (X_{k\Delta}^{\delta,h^{\delta}},\mathscr{L}_{X_{k\Delta}^{\delta}},Y_{r\varepsilon+k\Delta}^\delta)-\bar{b}(X_{k\Delta}^{\delta,h^{\delta}},\mathscr{L}_{X_{k\Delta}^{\delta}})\rangle\Big].
\end{align*}

For any $s\geq 0$ and any $\mathscr{F}_s$-measurable $\mathbb{R}^m$-valued random variable ${Y}$, we consider the following equation
\begin{equation}\label{n4.22}
	\left.\left\{\begin{aligned}
		&d\tilde{Y}_t=\frac{1}{\varepsilon}f(\tilde{Y}_t)dt+\frac{1}{\sqrt{\varepsilon}}g(\tilde{Y}_t)dW_t^2,~~~t\geq s,\\
		&\tilde{Y}_s=Y.
	\end{aligned}\right.\right.
\end{equation}
Then, by \cite[Theorem 4.2.4]{ll},  (\ref{n4.22}) has a unique solution denoted by $\tilde{Y}^{\delta,s,Y}_t$. By the definition of $Y^\delta _{t}$ in (\ref{n0}), for any $k\in \mathbb{N}$, it is clear that
\begin{equation*}
	Y_{t}^{\delta}=\tilde{Y}_{t}^{\delta,k\Delta,Y_{k\Delta}^{\delta}},~~~t\in[k\Delta,(k+1)\Delta].
\end{equation*}
Then we can obtain that
\begin{align}
\Psi_{k}(s,r):=\mathbb{E}\Big[\langle b^{\nu} (X_{k\Delta}^{\delta,h^{\delta}},\mathscr{L}_{X_{k\Delta}^{\delta}},\tilde{Y}_{s\varepsilon+k\Delta}^{\delta,k\Delta,Y_{k\Delta}^{\delta}})-\bar{b}(X_{k\Delta}^{\delta,h^{\delta}},\mathscr{L}_{X_{k\Delta}^{\delta}}),\nonumber\\
b^{\nu} (X_{k\Delta}^{\delta,h^{\delta}},\mathscr{L}_{X_{k\Delta}^{\delta}},\tilde{Y}_{r\varepsilon+k\Delta}^{\delta,k\Delta,Y_{k\Delta}^{\delta}})-\bar{b}(X_{k\Delta}^{\delta,h^{\delta}},\mathscr{L}_{X_{k\Delta}^{\delta}})\rangle\Big].\nonumber
\end{align}
It should be pointed out that for any fixed $y$, $\tilde{Y}_{s\varepsilon+k\Delta}^{\delta,k\Delta,y}$ is independent of $\mathscr{F}_{k\Delta},$ thus we get
\begin{align}
\Psi_{k}(s,r)=&\mathbb{E}\bigg\{\mathbb{E}\bigg[\big\langle b^{\nu} (X_{k\Delta}^{\delta,h^{\delta}},\mathscr{L}_{X_{k\Delta}^{\delta}},\tilde{Y}_{s\varepsilon+k\Delta}^{\delta,k\Delta,Y_{k\Delta}^{\delta}})-\bar{b}(X_{k\Delta}^{\delta,h^{\delta}},\mathscr{L}_{X_{k\Delta}^{\delta}}),\nonumber\\
&~~~~~~~~~b^{\nu} (X_{k\Delta}^{\delta,h^{\delta}},\mathscr{L}_{X_{k\Delta}^{\delta}},\tilde{Y}_{r\varepsilon+k\Delta}^{\delta,k\Delta,Y_{k\Delta}^{\delta}})-\bar{b}(X_{k\Delta}^{\delta,h^{\delta}},\mathscr{L}_{X_{k\Delta}^{\delta}})\big\rangle\big|\mathscr{F}_{k\Delta} \bigg]\bigg\}\nonumber\\
=&\mathbb{E}\bigg\{\mathbb{E}\bigg[\big\langle b^{\nu} (x,\mathscr{L}_{X_{k\Delta}^{\delta}},\tilde{Y}_{s\varepsilon+k\Delta}^{\delta,k\Delta,y})-\bar{b}(x,\mathscr{L}_{X_{k\Delta}^{\delta}}),\nonumber\\
&~~~~~~~~~b^{\nu} (x,\mathscr{L}_{X_{k\Delta}^{\delta}},\tilde{Y}_{r\varepsilon+k\Delta}^{\delta,k\Delta,y})-\bar{b}(x,\mathscr{L}_{X_{k\Delta}^{\delta}})\big\rangle\bigg]\bigg|_{(x,y)=(X_{k\Delta}^{\delta,h^{\delta}},Y_{k\Delta}^{\delta})} \bigg\}.\nonumber
\end{align}

Recall the definition of $\{\tilde{Y}_{s\varepsilon+k\Delta}^{\delta,k\Delta,y}\}_{s\geq0},$ by a time shift transformation it follows that
\begin{align}\label{z4.23}
\tilde{Y}_{s\varepsilon+k\Delta}^{\delta,k\Delta,y}=&y+\frac{1}{\varepsilon}\int_{k\Delta}^{s\varepsilon+k\Delta}f\big(\tilde{Y}_{r}^{\delta,k\Delta,y}\big)dr+\frac{1}{\sqrt{\varepsilon}}\int_{k\Delta}^{s\varepsilon+k\Delta}g\big(\tilde{Y}_{r}^{\delta,k\Delta,y}\big)dW_{r}^{2}\nonumber\\
=&y+\frac{1}{\varepsilon}\int_{0}^{s\varepsilon}f\big(\tilde{Y}_{r+k\Delta}^{\delta,k\Delta,y}\big)dr+\frac{1}{\sqrt{\varepsilon}}\int_{0}^{s\varepsilon}g\big(\tilde{Y}_{r+k\Delta}^{\delta,k\Delta,y}\big)dW_{r}^{2,k\Delta}\nonumber\\
=&y+\int_{0}^{s}f\big(\tilde{Y}_{r\varepsilon+k\Delta}^{\delta,k\Delta,y}\big)dr+\int_{0}^{s}g\big(\tilde{Y}_{r\varepsilon+k\Delta}^{\delta,k\Delta,y}\big)d\tilde{W}_{r}^{2,k\Delta},
\end{align}
where
\begin{equation}
\left\{W_r^{2,k\Delta}:=W_{r+k\Delta}^2-W_{k\Delta}^2\right\}_{r\ge0}\mathrm{~~and~~}\left\{\tilde{W}_r^{2,k\Delta}:=\frac{1}{\sqrt{\varepsilon}}W_{r\varepsilon}^{2,k\Delta}\right\}_{r\ge0}.\nonumber
\end{equation}

We recall that the process  $Y^y _s$ satisfies Eq.~(\ref{n2}), i.e.,
\begin{equation}\label{z4.24}
Y_s^y=y+\int_0^sf(Y_r^y)dr+\int_0^sg(Y_r^y)d\tilde{W}_r^2.
\end{equation}
Then, the uniqueness of solutions of (\ref{z4.23}) and (\ref{z4.24}) implies that $\{\tilde{Y}_{s\varepsilon+k\Delta}^{\delta,k\Delta,y}\}_{0\leq s \leq \frac{\Delta}{\varepsilon}}$ coincides in law with $\{Y_s^y\}_{0\leq s\leq\frac{\Delta}{\varepsilon}}$. Thus, using Markov and time-homogeneous properties of process $Y_s^y$ we have
\begin{align}
\Psi_{k}(s,r)=&\mathbb{E}\Big\{\mathbb{E}\Big[\langle b^{\nu} (x,\mathscr{L}_{X_{k\Delta}^{\delta}},Y_s^y)-\bar{b}(x,\mathscr{L}_{X_{k\Delta}^{\delta}}),\nonumber\\
&~~~~~~~~~b^{\nu} (x,\mathscr{L}_{X_{k\Delta}^{\delta}},Y_r^y)-\bar{b}(x,\mathscr{L}_{X_{k\Delta}^{\delta}})\rangle\Big]\Big|_{(x,y)=(X_{k\Delta}^{\delta,h^{\delta}},Y_{k\Delta}^{\delta})} \Big\}\nonumber\\
=&\mathbb{E}\left\{\mathbb{E}\Big[\langle\mathbb{E}\Big[\Big(b^{\nu} (x,\mathscr{L}_{X_{k\Delta}^{\delta}},Y_s^y)-\bar{b}(x,\mathscr{L}_{X_{k\Delta}^{\delta}})\Big)\Big|\mathscr{F}_{r}\right], \nonumber\\
&~~~~~~~~~b^{\nu} (x,\mathscr{L}_{X_{k\Delta}^{\delta}},Y_r^y)-\bar{b}(x,\mathscr{L}_{X_{k\Delta}^{\delta}})\rangle\Big]\Big|_{(x,y)=(X_{k\Delta}^{\delta,h^{\delta}},Y_{k\Delta}^{\delta})}\Big\} \nonumber\\
=&\mathbb{E}\left\{\mathbb{E}\Big[\langle\mathbb{E}\Big[b^{\nu} (x,\mathscr{L}_{X_{k\Delta}^{\delta}},Y_{s-r}^z)-\bar{b}(x,\mathscr{L}_{X_{k\Delta}^{\delta}})\right]\Big|_{\{z=Y_r^y\}},\nonumber \\
&~~~~~~~~~b^{\nu} (x,\mathscr{L}_{X_{k\Delta}^{\delta}},Y_r^y)-\bar{b}(x,\mathscr{L}_{X_{k\Delta}^{\delta}})\rangle\Big]\Big|_{(x,y)=(X_{k\Delta}^{\delta,h^{\delta}},Y_{k\Delta}^{\delta})}\Big\}. \nonumber
\end{align}
Therefore, according to Lemmas 4.2 and 4.4 in \cite{vvv} and Lemma 3.1 in \cite{cc}, we  get
\begin{align}\label{z4.25}
\Psi_{k}(s,r)\leq& C_{T}e^{-\frac{(s-r)\beta}{2}} \mathbb{E}\Big[1+|Y_r^{Y_{k\Delta}^{\delta}}|^2\Big] \nonumber\\
\leq&  C_{T}e^{-\frac{(s-r)\beta}{2}} \mathbb{E}\Big[1+|Y_{k\Delta}^{\delta}|^2\Big] \nonumber\\
\leq & C_{T} (1+|y|^2)e^{-\frac{(s-r)\beta}{2}},
\end{align}
where $\beta \in (0,\kappa)$.

Now, by inserting (\ref{z4.25}) into (\ref{z4.21}) we obtain
\begin{align*}
&\mathbb{E}\Big[\sup_{t\in[0,T]}{I}_{21}(t)\Big]\nonumber\\
\leq& C_T (1+|y|^2) \frac{\varepsilon^2}{\Delta^2}\left[\int_0^{\frac{\Delta}{\varepsilon}}\int_r^{\frac{\Delta}{\varepsilon}}e^{-\frac{(s-r)\beta}{2}}dsdr\right]\nonumber\\
\leq& C_T (1+|y|^2) \frac{\varepsilon^2}{\Delta^2}(\frac{2\Delta}{\beta \varepsilon}-\frac{4}{\beta ^2}+\frac{4}{\beta ^2}e^{-\frac{\beta \Delta}{2\varepsilon}})\nonumber\\
\leq& C_T (1+|y|^2) (\frac{\varepsilon^2}{\Delta^2} + \frac{\varepsilon}{\Delta} ).
\end{align*}
Consequently, in view of the assumption (\ref{2.11}), we can conclude that (\ref{e07}) holds.
\end{proof}

According to  the convergence (\ref{4.7}),  the following equality holds
\begin{align*}
	\tilde{X}_t &= x+\tilde{\mathcal{R}}_1(t)+\tilde{\mathcal{R}}_2(t)\quad\tilde{\mathbb{P}}\text{-a.s.}.
\end{align*}
Then following from (\ref{4.16}) and Lemmas \ref{l4.1}-\ref{l4.2}, we indicate that the limit $(X,\mathbf{P})$ of sequence
$\{(X^{\delta,h^{\delta}},\mathbf{P}^{\delta,\Delta})\}_{\delta>0}$ satisfies the following integral equation
\begin{align*}
	X_t &= x+\int_{\mathbb{R}^d\times\mathcal{Y}\times[0,t]}\bar{b}(X_s,\mathscr{L}_{\bar{X}_{s}})+\sigma(X_s,\mathscr{L}_{\bar{X}_{s}},y,\nu)P_1h\mathbf{P}(dhdyds)\\
	&=x+\int_{\mathbb{R}^d\times\mathcal Y\times[0,t]}\Phi(X_{s},\mathscr{L}_{\bar{X}_{s}},y,\nu,h)\mathbf{P}(dhdyds)\quad\mathbb{P}\text{-a.s.}.
\end{align*}  
 Therefore, $(X,\mathbf{P})$ satisfies (\ref{2.16}) in Definition \ref{d2.4}.

\vspace{1mm}

\noindent$\mathbf{Proof~of~(\ref{2.17}).}$ In this part, we aim to demonstrate that the second and third marginals of $\mathbf{P}$ are represented by the product of the invariant measure $\nu$ and the Lebesgue measure. This will be showed in Lemma \ref{l4.5}. Before doing that, we require several preliminary lemmas, namely Lemma \ref{l4.3} and \ref{l4.4} below.

Recall the controlled fast process
\begin{align*}
dY_t^{\delta,h^{\delta}}=&\frac{1}{\varepsilon}f(Y_t^{\delta,h^{\delta}})dt+\frac{1}{\sqrt{\delta\varepsilon}}g(Y_t^{\delta,h^{\delta}})h_t^{2,\delta}dt+\frac{1}{\sqrt{\varepsilon}}g(Y_t^{\delta, h^{\delta}})dW_t^2,~~~Y_0^{\delta,h^{\delta}}=y.
\end{align*}
Meanwhile, we also recall the uncontrolled fast process
\begin{equation}
d{Y}_t^{\delta}=\frac{1}{\varepsilon}f({Y}_t^{\delta})dt+\frac{1}{\sqrt{\varepsilon}}g({Y}_t^{\delta})dW_t^2,~~~{Y}_0^{\delta}=y.\nonumber
\end{equation}

In the following lemma, we show that the processes $Y_t^{\delta,h^{\delta}}$ and ${Y}_t^{\delta}$ are close in $\mathit{L}^2$-sense.

\begin{lemma}\label{l4.3}
	Let $M>0$, $\{h^{\delta}\}_{\delta>0}\subset\mathcal{A}_{M}$, $\Delta$ as in Definition (\ref{2.11}). Then
	\begin{equation}\label{4.17}
		\dfrac{1}{\Delta}\mathbb{E}\int_0^T|Y_t^{\delta,h^\delta}-{Y}_t^\delta|^2dt\to0,~~\text{as}~\delta\to0.
	\end{equation}
\end{lemma}
\begin{proof}
	Denote $\zeta _{t} := Y_t^{\delta,h^{\delta}}-{Y}_t^{\delta}$. Using It\^{o}'s
	formula for $|\zeta _{t}|^2$ and then taking expectation, we derive that
	\begin{align*}
		\frac{d}{dt}\mathbb{E}|\zeta_{t}|^{2}=& \frac{2}{\varepsilon}\mathbb{E}\Big[\langle f(Y_{t}^{\delta,h^{\delta}})-f({Y}_{t}^{\delta}),\zeta_{t}\rangle\Big] +\frac{1}{\varepsilon}\mathbb{E}\|g(Y_{t}^{\delta,h^{\delta}})-g({Y}_{t}^{\delta})\|^{2} \\
		&+\frac{2}{\sqrt{\delta\varepsilon}}\mathbb{E}\Big[\langle g(Y_{t}^{\delta,h^{\delta}})h_{t}^{2,\delta},\zeta_{t}\rangle\Big]\\
		=& \sum_{i=1}^3\mathcal{J}_i.
	\end{align*}
	Due to the condition (\ref{a111}), we can deduce that
	\begin{equation}\nonumber
		\mathcal{J}_1+\mathcal{J}_2\leq-\frac{\kappa}{\varepsilon}\mathbb{E}|\zeta_t|^2.
	\end{equation}
	Moreover, we have
	\begin{equation}\nonumber
		\mathcal{J}_3\leq\dfrac{C}{\sqrt{\delta\varepsilon}}\mathbb{E}\Big[|h_t^{2,\delta}||\zeta_t|\Big]\leq\dfrac{C}{\delta}\mathbb{E}|h_t^{2,\delta}|^2+\dfrac{\epsilon_0}{\varepsilon}\mathbb{E}|\zeta_t|^2,
	\end{equation}
	where $\epsilon_{0}\in(0,\kappa).$
	
	Then we have
	\begin{equation}\nonumber
		\frac{d}{dt}\mathbb{E}|\zeta_t|^2\leq-\frac{\beta}{\varepsilon}\mathbb{E}|\zeta_t|^2+\frac{C}{\delta}\mathbb{E}|h_t^{2,\delta}|^2,
	\end{equation}
	where $\beta:=\kappa-\epsilon_{0}>0.$ The Gronwall's lemma implies that
	\begin{equation}\nonumber
		\mathbb{E}|\zeta_t|^2\leq\dfrac{C}{\delta}\int_0^te^{-\frac{\beta}{\varepsilon}(t-s)}\mathbb{E}|h_s^{2,\delta}|^2ds.
	\end{equation}
	Thus, for $h^\delta\in\mathcal{A}_M,$ we have
	\begin{equation}\nonumber
		\mathbb{E}\int_0^T|\zeta_t|^2dt\leq\frac{C}{\delta}\mathbb{E}\Big[\int_0^T\int_0^te^{-\frac{\beta}{\varepsilon}(t-s)}|h_s^{2,\delta}|^2dsdt\Big]\leq\frac{C_{M,T}\varepsilon}{\delta}.
	\end{equation}
	Note that $\frac{\varepsilon}{\delta \Delta} \to 0$~as~$\delta \to 0$, we complete the proof of (\ref{4.17}).
\end{proof}

Correspondingly, for $s \geq t$ we can establish the two parameter process $Y (s;t)$ which solves
\begin{align*}
	dY(s;t)=\frac{1}{\varepsilon}f(Y(s;t))ds+\frac{1}{\sqrt{\varepsilon}}g(Y (s;t))dW_s^2,~~
	Y (t;t)={Y}_t^{\delta}.\nonumber
\end{align*}
As demonstrated by the forthcoming lemma, the process ${Y}_s^{\delta}$ is close to the process $Y (s;t)$ in $\mathit{L}^2$-sense on the interval $s \in [t,t+\Delta]$.

\begin{lemma}\label{l4.4}
Let $M>0$, $\{h^{\delta}\}_{\delta>0}\subset\mathcal{A}_{M}$, $\Delta$ as in Definition (\ref{2.11}). Then
\begin{equation}\label{4.18}
\dfrac{1}{\Delta}\mathbb{E}\int_t^{t+\Delta}|{Y}_s^\delta-Y (s;t)|^2ds\to0,~~\text{as}~\delta\to0.
\end{equation}
\end{lemma}
\begin{proof}
The proof is omitted since one can follow the similar argument as in Lemma \ref{l4.3} to obtain (\ref{4.18}).
\end{proof}

We are now in the position to show that (\ref{2.17}) holds.
\begin{lemma}\label{l4.5}
$\mathbf{P}$ has the decomposition (\ref{2.17}), i.e., for any $\phi \in C_b(\mathcal{Y}),$
\begin{equation*}
\int_{\mathbb{R}^d\times\mathcal{Y}\times[0,T]}\phi(y)\mathbf{P}(dhdydt)=\int_{0}^{T}\int_{\mathcal{Y}}\phi(y)\nu(dy)dt.
\end{equation*}
\end{lemma}
\begin{proof}
Without loss of generality, we assume that $\phi$ is Lipschitz continuous. Note that we have the following decomposition
\begin{align*}
&\int_{\mathbb{R}^d\times\mathcal Y\times[0,T]}\phi(y)\mathbf{P}^{\delta,\Delta}(dhdydt)\\
=& \int_0^T\frac{1}{\Delta}\int_t^{t+\Delta}\phi(Y_s^{\delta,h^\delta})dsdt\\
=&\Bigg(\int_0^T\frac{1}{\Delta}\int_t^{t+\Delta}\phi(Y_s^{\delta,h^\delta})dsdt-\int_0^T\frac{1}{\Delta}\int_t^{t+\Delta}\phi({Y}_s^\delta)dsdt\Bigg) \\
&+\Bigg(\int_{0}^{T}\frac{1}{\Delta}\int_{t}^{t+\Delta}\phi({Y}_{s}^{\delta})dsdt-\int_{0}^{T}\frac{1}{\Delta}\int_{t}^{t+\Delta}\phi(Y (s;t))dsdt\Bigg) \\
&+\Bigg(\int_{0}^{T}\frac{1}{\Delta}\int_{t}^{t+\Delta}\phi(Y(s;t))dsdt-\int_{0}^{T}\int_{\mathcal{Y}}\phi(y)\nu(dy)dt\Bigg) \\
&+\int_{0}^{T}\int_{\mathcal{Y}}\phi(y)\nu(dy)dt\\
=:&\sum_{i=1}^4 \mathcal{O}_i^\delta.
\end{align*}

Our aim is to demonstrate that the terms $\mathcal{O}_{i}^{\delta}, i=1,\ldots,3,$  converge to zero in probability as $\delta \to 0$. In light of Lemma \ref{l4.3} and the dominated convergence theorem, we infer that $\mathcal{O}_{1}^{\delta}$ also tends to zero in probability as $\delta \to 0$.  Correspondingly, Lemma \ref{l4.4} and the dominated convergence theorem indicate that $\mathcal{O}_{2}^{\delta}$ tends to zero in probability.

Now, let's address the term  $\mathcal{O}_{3}^{\delta}$. We introduce the time-rescaled process  $\tilde Y_{s}:=Y (t+\varepsilon s;t)$ given by
\begin{equation}\nonumber
d\tilde Y_{s}=f(\tilde Y_{s})ds+g(\tilde Y_{s})dW_{s}^{2},~~\tilde Y_{0}={Y}_{t}^{\delta},\quad0\leq s\leq\frac{\Delta}{\varepsilon}.
\end{equation}
We notice that
\begin{equation}\nonumber
\dfrac{1}{\Delta}\int_t^{t+\Delta}\phi(Y(s;t))ds=\dfrac{\varepsilon}{\Delta}\int_0^{\frac{\Delta}{\varepsilon}}\phi(\tilde Y_s)ds.
\end{equation}
Thus, by the strong dissipation assumption (\ref{a111}) and the scale condition (\ref{2.11}) and making use of the classical Birkhoff ergodic theorem (cf.~\cite{B31} and see also \cite[Theorem 1.1]{D23}), we can obtain 
\begin{equation}\nonumber
\lim_{\delta\to0}\frac{\varepsilon}{\Delta}\int_{0}^{\frac{\Delta}{\varepsilon}}\phi(\tilde Y_{s})ds=\int_{\mathcal{Y}}\phi(y)\nu(dy),
\end{equation}
which together with the dominated convergence theorem implies  that  $\mathcal{O}_{3}^{\delta}$ tends to zero in probability.

Consequently, we deduce that as $\delta\rightarrow 0$, we have the following convergence
$$\int_{\mathbb{R}^d\times\mathcal Y\times[0,T]}\phi(y)\mathbf{P}^{\delta,\Delta}(dhdydt) \xrightarrow{\mathbb{P}} \int_0^T\int_{\mathcal{Y}}\phi(y)\nu(dy)dt.$$
On the other hand, since $\mathbf{P}^{\delta,\Delta} \Rightarrow \mathbf{P}$ weakly as $\delta \rightarrow 0$, the continuous mapping theorem yields that
$$\int_{\mathbb{R}^d\times\mathcal Y\times[0,T]}\phi(y)\mathbf{P}^{\delta,\Delta}(dhdydt) \xrightarrow{\text{Law}} \int_{\mathbb{R}^d\times\mathcal Y\times[0,T]}\phi(y)\mathbf{P} (dhdydt).$$
By the uniqueness of the limit, we conclude
$$\int_{\mathbb{R}^d\times\mathcal Y\times[0,T]}\phi(y)\mathbf{P} (dhdydt)=\int_0^T\int_{\mathcal{Y}}\phi(y)\nu(dy)dt,\quad \mathbb{P}\text{-a.s.}$$
Then, we finish the proof of (\ref{2.17}).
\end{proof}

\subsection{Lower bound of Laplace principle }
In this subsection, we  establish the lower bound of the Laplace principle. Specifically, we aim to demonstrate that for all bounded and continuous functions $\Lambda:C([0,T];\mathbb{R}^{n})\to\mathbb{R}$,
\begin{align}
&\liminf_{\delta\to0}\Bigg(-{\delta}\log\mathbb{E}\left[\exp\left\{-\frac{1}{\delta}\Lambda(X^{\delta})\right\}\right]\Bigg) \nonumber\\
\geq&\operatorname*{inf}_{(\varphi,\mathbf{P})\in\mathcal{V}_{(\Phi,\nu,x, \bar{X} )}}\left[{\frac{1}{2}}\int_{\mathbb{R}^d\times\mathcal{Y}\times[0,T]}|h|^{2}\mathbf{P}(dhdydt)+\Lambda(\varphi)\right].\nonumber
\end{align}
We only need to  prove the lower limit along any subsequence for which
\begin{equation*}
-{\delta}\log\mathbb{E}\left[\exp\left\{-\frac{1}{\delta}\Lambda(X^{\delta})\right\}\right],
\end{equation*}
converges. Note that such a subsequence  exists since
\begin{align*}
\left|-{\delta}\log\mathbb{E}\left[\exp\left\{-\frac{1}{\delta}\Lambda(X^{\delta})\right\}\right]\right|\leq C\|\Lambda\|_\infty.
\end{align*}
In light of  \cite[Theorem 3.17]{f}, for any $\eta >0$, there exists $M>0$ such that for any $\delta>0$, there exists $h^{\delta}\in\mathcal{A}_{M},$
we get
\begin{equation*}
-{\delta}\log\mathbb{E}\left[\exp\left\{-\frac{1}{\delta}\Lambda(X^\delta)\right\}\right]\geq\mathbb{E}\left[\frac{1}{2}\int_0^T|h_s^{\delta}|^2ds+\Lambda(X^{\delta,h^\delta})\right]-\eta.
\end{equation*}

Therefore, if we utilize this control $h^{\delta}$ and the associated controlled process $X^{\delta,h^\delta}$  to construct occupation measures $\mathbf{P}^{\delta,\Delta}$,  according to Proposition \ref{p4.1}, for any sequence in $\{\delta\}$ there exists a subsequence still denoted by $\{\delta\}$ such that
\begin{equation*}
(X^{\delta,h^\delta},\mathbf{P}^{\delta,\Delta})\Rightarrow(X,\mathbf{P}),
\end{equation*}
with $(X,\mathbf{P})\in\mathcal{V}_{(\Phi,\nu,x, \bar{X} )}.$ Using Fatou's lemma, it follows that
\begin{align*}
&\liminf_{\delta\to0}\Bigg(-\delta\log\mathbb{E}\left[\exp\left\{-\frac{1}{\delta}\Lambda(X^{\delta})\right\}\right]\Bigg) \\
\geq&\liminf_{\delta\to0}\mathbb{E}\left[\frac{1}{2}\int_{0}^{T}|h_{s}^{\delta}|^{2}ds+\Lambda(X^{\delta,h^{\delta}})\right]-\eta  \\
\geq&\liminf_{\delta\to0}\mathbb{E}\left[\frac{1}{2}\int_{0}^{T}\frac{1}{\Delta}\int_{t}^{t+\Delta}|h_{s}^{\delta}|^{2}dsdt+\Lambda(X^{\delta,h^{\delta}})\right]-\eta  \\
=&\liminf_{\delta\to0}\mathbb{E}\left[\frac{1}{2}\int_{\mathbb{R}^d\times\mathcal{Y}\times[0,T]}|h|^{2}\mathbf{P}^{\delta,\Delta}(dhdydt)+\Lambda(X^{\delta,h^{\delta}})\right]-\eta  \\
\geq&\mathbb{E}\left[\frac{1}{2}\int_{\mathbb{R}^d\times\mathcal{Y}\times[0,T]}|h|^{2}\mathbf{P}(dhdydt)+\Lambda(X)\right]-\eta  \\
\geq&\inf_{(\varphi,\mathbf{P})\in\mathcal{V}_{(\Phi,\nu,x, \bar{X} )}}\left[\frac{1}{2}\int_{\mathbb{R}^d\times\mathcal{Y}\times[0,T]}|h|^{2}\mathbf{P}(dhdydt)+\Lambda(\varphi)\right]-\eta  \\
\geq&\operatorname*{inf}_{\varphi\in C([0,T];\mathbb{R}^{n})}\left[I(\varphi)+\Lambda(\varphi)\right]-\eta.
\end{align*}
Since $\eta>0$ is arbitrary, the lower bound is proved. \hspace{\fill}$\Box$

\subsection{Compactness of level sets of $I(\cdot)$}
In this part,
our objective is to establish that for each $s<\infty$, the level set
\begin{equation*}
\Gamma(s):=\Big\{\varphi\in C([0,T];\mathbb{R}^n):I(\varphi)\leq s\Big\}
\end{equation*}
is a compact subset in $C([0,T];\mathbb{R}^n)$. More precisely,  we show the pre-compactness of $\Gamma(s)$ in Lemma $\ref{l4.6}$, and  demonstrate that it is closed in Lemma $\ref{l4.8}$. Then, we have the desired result.

\begin{lemma}\label{l4.6}
Fix $K < \infty $ and consider any sequence $\{(\varphi^{k},\mathbf{P}^{k})\}_{k\in\mathbb{N}}\subset\mathcal{V}_{(\Phi,\nu,x, \bar{X} )}$ such that for any $k\in \mathbb{N}$, $(\varphi^{k},\mathbf{P}^{k})$ is viable and
\begin{equation}\label{z4.31}
\int_{\mathbb{R}^d\times\mathcal{Y}\times[0,T]}\big(|h|^2+|y|^2\big)\mathbf{P}^k(dhdydt)\leq K.
\end{equation}
Then, $\{(\varphi^{k},\mathbf{P}^{k})\}_{k\in\mathbb{N}}$ is pre-compact.
\end{lemma}
\begin{proof}
For any $0\leq t_{1}<t_{2}\leq T$ and  $k\in \mathbb{N}$,
\begin{align*}
|\varphi_{t_2}^k-\varphi_{t_1}^k| =&\left|\int_{\mathbb{R}^d\times\mathcal{Y}\times[t_{1},t_{2}]}\Phi(\varphi_{s}^{k},\mathscr{L}_{\bar{X}_{s}},y,\nu,h)\mathbf{P}^{k}(dhdyds)\right|  \\
\leq& C(t_2-t_1)^{\frac{1}{2}}\Bigg(\int_{\mathbb{R}^d\times\mathcal{Y}\times[t_1,t_2]}|\bar{b}(\varphi_s^k,\mathscr{L}_{\bar{X}_s})+\sigma(\varphi_s^k,\mathscr{L}_{\bar{X}_s},y,\nu)h^1|^2\mathbf{P}^k(dhdyds)\Bigg)^{\frac{1}{2}}\\
\leq& C(t_2-t_1)^{\frac{1}{2}}.
\end{align*}
This combined with the fact that  $\varphi_0^k = x$ gives the pre-compactness of $\{\varphi^{k}\}_{k\in\mathbb{N}}$  by Arzel\`{a}-Ascoli theorem.

The pre-compactness of $\{\mathbf{P}^{k}\}_{k\in\mathbb{N}}$ is inferred from  (\ref{z4.31}) by employing the same argument as in Proposition \ref{p4.1}.
\end{proof}

To demonstrate that the level set  $\Gamma(s)$ is closed, we introduce a crucial lemma stating that the limit of any sequence of viable pairs is also viable.
\begin{lemma}\label{l4.7}
Fix $K < \infty $ and consider any sequence $\{(\varphi^{k},\mathbf{P}^{k})\}_{k\in\mathbb{N}}$ such that for any $k\in \mathbb{N}$, $(\varphi^{k},\mathbf{P}^{k})$ is viable and
\begin{equation}\label{z4.32}
	\int_{\mathbb{R}^d\times\mathcal{Y}\times[0,T]}\big(|h|^2+|y|^2\big)\mathbf{P}^k(dhdydt)\leq K.
\end{equation}
Then, the limit $(\varphi,\mathbf{P})$ is a viable pair.
\end{lemma}
\begin{proof}
Since $(\varphi^{k},\mathbf{P}^{k})$ is viable, we know that
\begin{equation*}
\varphi_{t}^k =x+\int_{\mathbb{R}^d\times\mathcal{Y}\times[0,t]}\Phi(\varphi_{s}^{k},\mathscr{L}_{\bar{X}_{s}},y,\nu,h)\mathbf{P}^{k}(dhdyds),
\end{equation*}
for every $t \in [0,T]$, and
\begin{equation*}
\mathbf{P}^k(dhdydt)=\eta^k(dh|y,t)\nu(dy)dt,
\end{equation*}
where $\eta^k$ is a sequence of stochastic kernels.

First, by applying Fatou's lemma, it is easy to show that $\mathbf{P}$ has a finite second moment, as required by condition (i) in Definition \ref{d2.4}. Additionally, note that the function $\Phi(\varphi,\mu,y,\nu,h)$ is continuous w.r.t.~$\varphi,\mu,y,\nu$ and affine in $h$. Furthermore, the uniform integrability of $\mathbf{P}^k$ can be demonstrated analogously to the argument presented in Proposition \ref{p4.1}. Therefore, using assumption (\ref{z4.32}) and considering the convergences $\mathbf{P}^k \to \mathbf{P}$ and $\varphi^k \to \varphi$, we conclude that $(\varphi,\mathbf{P})$ satisfies (\ref{2.16}). For the same reasons, it is evident that $\mathbf{P}$ satisfies (\ref{2.17}). This completes the proof.
\end{proof}

\begin{lemma}\label{l4.8}
The functional $I(\varphi)$ is lower semicontinuous.
\end{lemma}
\begin{proof}
Consider a sequence $\varphi^k$ with limit $\varphi$. We intend to prove
\begin{equation*}
\liminf_{k\to\infty}I(\varphi^k)\geq I(\varphi).
\end{equation*}
It is enough to focus on the case when $I(\varphi^k)$ has a finite limit, i.e., there exists a $M < \infty$ such that $\lim_{k\to\infty}I(\varphi^k)\leq M$.
We recall the definition
\begin{equation*}
	I(\varphi):=\inf_{(\varphi,\mathbf{P})\in\mathcal{V}_{(\Phi,\nu,x, \bar{X} )}}\left\{\frac{1}{2}\int_{\mathbb{R}^d\times\mathcal{Y}\times[0,T]}|h|^{2} \mathbf{P}(dhdydt)\right\}.
\end{equation*}
Then, there exists a sequence $\{\mathbf{P}^{k}\}_{k\in\mathbb{N}}$ such that $(\varphi^{k},\mathbf{P}^{k})\subset\mathcal{V}_{(\Phi,\nu,x, \bar{X} )}$
and
\begin{equation*}
\sup_{k\in\mathbb{N}}\frac{1}{2}\int_{\mathbb{R}^d\times\mathcal{Y}\times[0,T]}\big(|h|^{2}+|y|^{2}\big)\mathbf{P}^{k}(dhdydt)\leq M+1+\frac{1}{2}\int_{\mathcal{Y}}|y|^{2}\nu(dy),
\end{equation*}
and such that
\begin{equation}\label{e08}
I(\varphi^k)\geq\frac{1}{2}\int_{\mathbb{R}^d\times\mathcal{Y}\times[0,T]}|h|^2\mathbf{P}^k(dhdydt)-\frac{1}{k}.
\end{equation}
According to \cite[(4.28)]{aa} or \cite[Theorem 4.3.9]{ll}, we know that $\int_{\mathcal{Y}}|y|^{2}\nu(dy)<\infty$. Thus, there exists a constant $M'> 0$ such that
\begin{equation}\label{e09}
	\sup_{k\in\mathbb{N}}\frac{1}{2}\int_{\mathbb{R}^d\times\mathcal{Y}\times[0,T]}|h|^{2}\mathbf{P}^{k}(dhdydt)\leq M'.
\end{equation}
 In view of Lemma \ref{l4.6}, we can consider a subsequence along which $(\varphi^{k},\mathbf{P}^{k})$ converges to a limit $(\varphi,\mathbf{P})$. From Lemma \ref{l4.7}, we know that $(\varphi,\mathbf{P})$ is viable. Therefore, by (\ref{e08})-(\ref{e09}) and Fatou's lemma
\begin{align*}
\liminf_{k\to\infty} I(\varphi^{k})& \geq\liminf_{k\to\infty}\Bigg(\frac{1}{2}\int_{\mathbb{R}^d\times\mathcal{Y}\times[0,T]}|h|^{2}\mathbf{P}^{k}(dhdydt)-\frac{1}{k}\Bigg)  \\
&\geq\frac{1}{2}\int_{\mathbb{R}^d\times\mathcal{Y}\times[0,T]}|h|^{2}\mathbf{P}(dhdydt) \\
&\geq\inf_{(\varphi,\mathbf{P})\in\mathcal{V}_{(\Phi,\nu,x, \bar{X} )}}\left\{\frac{1}{2}\int_{\mathbb{R}^d\times\mathcal{Y}\times[0,T]}|h|^{2}\mathbf{P}(dhdydt)\right\} \\
&=I(\varphi),
\end{align*}
which concludes the proof of lower-semicontinuity of $I$.
\end{proof}

\subsection{Upper bound of Laplace principle }\label{sec5.4}
The proof of  the upper bound  of Laplace principle is more complicated than the lower bound, where we  need to construct the feedback controls to achieve the bound.

Recall the definition (\ref{sulv}) in Theorem \ref{t2.1}, the rate function is given by the following 
\begin{equation*}
I(\varphi)=\inf_{\mathbf{P}\in\Xi_{\varphi}}\frac{1}{2}\int_{0}^{T}\int_{\mathbb{R}^d\times\mathcal{Y}}|h|^{2}\mathbf{P}_{t}(dhdy)dt,
\end{equation*}
where the set
\begin{align*}
\left.\Xi_{\varphi}:=\left\{\begin{aligned}\mathbf{P}&:[0,T]\to\mathcal{P}(\mathbb{R}^d\times\mathcal{Y}):\\
&\text{$\exists$ a stochastic kernel $\eta$ such that }\mathbf{P}_t(A_1\times A_2)=\int_{A_2}\eta(A_1|y,t)\nu(dy),\\
&\int_0^T\int_{\mathbb{R}^d\times\mathcal{Y}}(|h|^2+|y|^2)\mathbf{P}_t(dhdy)dt<\infty,\\
&\varphi_t=x+\int_0^t\int_{\mathbb{R}^d\times\mathcal{Y}}\Phi(\varphi_s,\mathscr{L}_{\bar{X}_s},y,\nu,h)\mathbf{P}_s(dhdy)ds.\end{aligned}\right.\right\}.
\end{align*}
Moreover, we also define
\begin{equation}\label{z4.35}
\tilde{I}(\varphi):=\inf_{z\in\tilde{\Xi}_\varphi}\frac{1}{2}\int_0^T\int_\mathcal{Y}|z_t(y)|^2\nu(dy)dt,
\end{equation}
where the set
\begin{align*}
\left.\tilde{\Xi}_{\varphi}:=\left\{\begin{aligned}z:~&[0,T]\times\mathcal{Y}\to\mathbb{R}^d:\\
	&\int_0^T\int_\mathcal{Y}(|z_t(y)|^2+|y|^2)\nu(dy)dt<\infty,\\
	&\varphi_t=x+\int_0^t\int_\mathcal{Y}\Phi(\varphi_s,\mathscr{L}_{\bar{X}_s},y,\nu,z_s (y))\nu(dy)ds.\end{aligned}\right.\right\}.
\end{align*}

\begin{lemma}
$I(\varphi)=\tilde{I}(\varphi), ~\varphi\in C([0,T];\mathbb{R}^n)$.
\end{lemma}

\begin{proof}
First, for any given $z \in \tilde{\Xi}_{\varphi}$, we can define $\mathbf{P} \in {\Xi}_{\varphi}$ by
\begin{equation*}
\mathbf{P}_t(dhdy):=\delta_{z_{t}(y)}(dh)\nu(dy).
\end{equation*}
Therefore, it is evident that
\begin{equation*}
I(\varphi)\leq\tilde{I}(\varphi).
\end{equation*}
Conversely, for any given $\mathbf{P} \in {\Xi}_{\varphi}$, we define
\begin{equation*}
z_t(y):=\int_{\mathbb{R}^d}h\eta(dh|y,t),
\end{equation*}
where $\eta(dh|y,t)$ is the conditional distribution of $\mathbf{P}$. Since the mapping $\Phi$ is affine with respect to the control $h$, we have $z \in \tilde{\Xi}_{\varphi}$. Then applying Jensen's inequality we can obtain that
\begin{align*}
\int_{0}^{T}\int_{\mathbb{R}^d\times\mathcal{Y}}|h|^{2}\mathbf{P}_{t}(dhdy)dt& \geq\int_{0}^{T}\int_{\mathcal{Y}}\left|\int_{\mathbb{R}^d}h\eta(dh|y,t)\right|^{2}\nu(dy)dt  \\
&=\int_0^T\int_{\mathcal{Y}}|z_t(y)|^2\nu(dy)dt.
\end{align*}
Then we can deduce that
\begin{equation*}
	I(\varphi)\geq\tilde{I}(\varphi).
\end{equation*}
Thus, the lemma follows.
\end{proof}

The next lemma provides an explicit representation of the infimization problem (\ref{z4.35}), which is crucial for proving the upper bound of Laplace principle.

\begin{lemma}
The control $h:[0,T]\times\mathcal{Y}\to\mathbb{R}^d$ defined by
\begin{align*}
&h_t (y):=P_1^*\sigma^*(\varphi_{t},\mathscr{L}_{\bar{X}_{t}},y,\nu)Q^{-1}(\varphi_{t},\mathscr{L}_{\bar{X}_t},\nu)(\dot{\varphi}_t-\bar{b}(\varphi_t,\mathscr{L}_{\bar{X}_t})),
\end{align*}
attains the infimum in (\ref{z4.35}), where
\begin{equation*}
Q(\varphi_{t},\mathscr{L}_{\bar{X}_{t}},\nu):=\int_{\mathcal{Y}} \sigma(\varphi_{t},\mathscr{L}_{\bar{X}_{t}},y,\nu)P_1P_1^*\sigma^*(\varphi_{t},\mathscr{L}_{\bar{X}_{t}},y,\nu)\nu(dy).
\end{equation*}
Furthermore, the infimization problem (\ref{z4.35}) has the explicit solution
\begin{align}\label{z4.36} \tilde{I}(\varphi)=\begin{cases}\frac{1}{2}\int_0^T|Q^{-1/2}(\varphi_{t},\mathscr{L}_{\bar{X}_t},\nu)(\dot{\varphi}_t-\bar{b}(\varphi_t,\mathscr{L}_{\bar{X}_t}))|^2dt,&\varphi(0)=x,\varphi ~\text{is absolutely continuous},\\+\infty,&\text{otherwise},\end{cases}
\end{align}
\end{lemma}
\begin{proof}
For any $z_t \in \tilde{\Xi}_{\varphi}$, it follows that
\begin{align*}
\dot{\varphi}_{t}=\int_{\mathcal{Y}}\Phi(\varphi_{t},\mathscr{L}_{\bar{X}_{t}},y,\nu,z_t(y))\nu(dy)=\bar{b}(\varphi_{t},\mathscr{L}_{\bar{X}_{t}})+\int_{\mathcal{Y}}\sigma(\varphi_{t},\mathscr{L}_{\bar{X}_{t}},y,\nu)z^1_t(y) \nu(dy),~{\varphi}_{0}=x.
\end{align*}
where $z^1_t(y) := P_1 z_t(y)$.
Then, utilizing the H{\"o}lder inequality for integrals of matrices (cf. \cite[Lemma 5.1]{o}) gives 
\begin{equation*}
\int_0^T\int_\mathcal{Y}|z_t(y)|^2\nu(dy)dt\geq\int_0^T(\dot{\varphi}_t-\bar{b}(\varphi_t,\mathscr{L}_{\bar{X}_t}))^*Q^{-1}(\varphi_{t},\mathscr{L}_{\bar{X}_t},\nu)(\dot{\varphi}_t-\bar{b}(\varphi_t,\mathscr{L}_{\bar{X}_t}))dt.
\end{equation*}
Furthermore, for any $t \in [0,T]$, if we take
\begin{equation}\label{z4.37}
h_t(y):=P_1^*\sigma^*(\varphi_{t},\mathscr{L}_{\bar{X}_{t}},y,\nu)Q^{-1}(\varphi_{t},\mathscr{L}_{\bar{X}_t},\nu)(\dot{\varphi}_t-\bar{b}(\varphi_t,\mathscr{L}_{\bar{X}_t})),
\end{equation}
then $h\in\tilde{\Xi}_\varphi $ and
\begin{equation*}
\int_0^T\int_\mathcal{Y}|h_t(y)|^2\nu(dy)dt=\int_0^T(\dot{\varphi}_t-\bar{b}(\varphi_t,\mathscr{L}_{\bar{X}_t}))^*Q^{-1}(\varphi_{t},\mathscr{L}_{\bar{X}_t},\nu)(\dot{\varphi}_t-\bar{b}(\varphi_t,\mathscr{L}_{\bar{X}_t}))dt,
\end{equation*}
which yields that (\ref{z4.36}) holds and the infimum of (\ref{z4.35}) is achieved in $h$ defined by (\ref{z4.37}).
\end{proof}

We are now in the position to prove the Laplace principle upper bound and thus complete the proof of Theorem \ref{t2.1}. Our goal is to demonstrate that for all bounded, continuous functions $\Lambda$ mapping $C([0,T]; \mathbb{R}^n)$ into $\mathbb{R}$, we have
\begin{align}\label{z4.38}
&\limsup_{\delta\to0}\Bigg(-{\delta}\log\mathbb{E}\left[\exp\left\{-\frac{1}{\delta}\Lambda(X^\delta)\right\}\right]\Bigg)
\nonumber\\
&\leq\inf_{\varphi\in C([0,T];\mathbb{R}^n)}\left[I(\varphi)+\Lambda(\varphi)\right]
\nonumber\\
&=\inf_{\varphi\in C([0,T];\mathbb{R}^n)}\left[\tilde{I}(\varphi)+\Lambda(\varphi)\right].
\end{align}
Notice that for any $\eta > 0$, there exists $\psi \in C([0,T];\mathbb{R}^n)$ with $\psi_0 = x$ such that
\begin{equation}\label{z4.39}
\tilde{I}(\psi)+\Lambda(\psi)\leq\inf_{\varphi\in C([0,T];\mathbb{R}^n)}\left[\tilde{I}(\varphi)+\Lambda(\varphi)\right]+\eta<\infty,
\end{equation}
and for every $z_t \in \tilde{\Xi}_{\psi}$,
\begin{equation}\label{z4.40}
\psi_t=x+\int_0^t\int_\mathcal{Y}\Phi(\psi_s,\mathscr{L}_{\bar{X}_s},y,\nu,z_s(y))\nu(dy)ds.
\end{equation}
Note that $\Lambda$ is bounded, it implies $\tilde{I}(\psi) < \infty$, and therefore, $\psi$ is absolutely continuous by the definition of $\tilde{I}$. For this specific function $\psi$, we define $\bar{h}_t(y)$ given by
\begin{equation*}
\bar{h}_t (y):=P_1^*\sigma^*(\psi_{t},\mathscr{L}_{\bar{X}_{t}},y,\nu)Q^{-1}(\dot{\psi}_t-\bar{b}(\psi_t,\mathscr{L}_{\bar{X}_t})),
\end{equation*}
then we have $\bar{h}.(y)\in L^{2}([0,T];\mathbb{R}^d)$ uniformly in $y \in \mathcal{Y} $. From a standard mollification argument, we can, without loss of generality, assume that
\begin{align}\label{z4.42}
	\bar{h}\text{~is Lipschitz continuous in}~t\in[0,T].
\end{align}
Indeed, let $0\leq\chi\in C_0^{\infty}(\mathbb{R})$ with support contained in $\{r:|r|\leq1\}$ such that $\int_\mathbb{R}\chi(r)dr=1,$ and for any $k \geq1$, let $\chi_{k}(r):=k\chi(kr)$ and define
\begin{equation*}
	\bar{h}_t^k(y):=\int_\mathbb{R}\bar{h}_r(y)\chi_k(t-r)dr.
\end{equation*}
By Subsection 4.3 in \cite{ff} (more precisely the forth and seventh display in Page 4764 in \cite{ff}), due to the property of convolutions, it clear  that for any $t_1,t_2 \in  [0,T],$
\begin{equation*}
|\bar{h}_{t_1}^k(y)-\bar{h}_{t_2}^k(y)|\leq c_k|t_1-t_2|,~~y\in\mathcal{Y}
\end{equation*}
and
\begin{equation*}
\|\bar{h}_.^k(y)-\bar{h}_.(y)\|_{L^2([0,T];\mathbb{R}^d)}\to0,~k\to\infty,\text{~uniformly in~}y\in\mathcal{Y}.
\end{equation*}
Furthermore, by  $(\mathbf{A_1})$-$(\mathbf{A_2})$ we  deduce that
\begin{align}\label{z4.43}
\bar{h}\text{~is Lipschitz continuous and bounded in~}y\in\mathcal{Y}.
\end{align}
Thus, by (\ref{z4.42}) and (\ref{z4.43}), we can also conclude that the same properties hold for the function
\begin{equation*}
\phi(\cdot,\cdot):=|\bar{h}_.(\cdot)|^2:[0,T]\times\mathcal{Y}\to\mathbb{R}.
\end{equation*}

Now we define a control by feedback form
\begin{equation*}
\bar h_t^\delta:=\bar h_t(Y_t^\delta).
\end{equation*}
By employing Khasminskii's time discretization scheme, we can establish the following convergence
\begin{equation}\label{z4.44}
\lim\limits_{\delta\to0}\mathbb{E}\int_0^T\phi(t,{Y}_t^\delta)dt=\int_0^T\int_\mathcal{Y}\phi(t,y)\nu(dy)dt,
\end{equation}
whose proof can refer to \cite[Subsection 6.4 in Appendix]{bb}.
Additionally, let $\psi\in C([0,T];\mathbb{R}^n)$ be the unique solution to the control problem (\ref{z4.40}) with $\bar h_t (y)$, we can infer that
\begin{equation}\label{z4.45}
X^{\delta,\bar{h}^\delta}\to\psi\quad\text{in~}C([0,T];\mathbb{R}^n)~~~\mathbb{P}\text{-a.s.},~\text{~as~}\delta\to0,
\end{equation} 
whose proof is postponed in Section \ref{s8.1} in Appendix.

In the following, we are able to prove (\ref{z4.38}). By (\ref{z4.36}), (\ref{z4.39}), (\ref{z4.44}) and (\ref{z4.45}), we have
\begin{align*}
&\limsup_{\delta\to0}\Bigg(-{\delta}\operatorname{log}\mathbb{E}\left[\operatorname{exp}\left\{-\frac{1}{\delta}\Lambda(X^{\delta})\right\}\right]\Bigg) \nonumber\\
=&\limsup_{\delta\to0}\operatorname*{inf}_{h\in\mathcal{A}}\mathbb{E}\left[\frac{1}{2}\int_{0}^{T}|h_{s}|^{2}ds+\Lambda(X^{\delta,h})\right] \nonumber\\
\leq&\limsup_{\delta\to0}\mathbb{E}\left[\frac{1}{2}\int_{0}^{T}|\bar{h}_{s}^{\delta}|^{2}ds+\Lambda(X^{\delta,\bar{h}^{\delta}})\right] \nonumber \\
=&\mathbb{E}\left[\frac{1}{2}\int_{0}^{T}\int_{\mathcal{Y}}|\bar{h}_{s}(y)|^{2}\nu(dy)ds+\Lambda(\psi)\right] \nonumber\\
=&I(\psi)+\Lambda(\psi) \nonumber\\
\leq&\operatorname*{inf}_{\varphi\in C([0,T];\mathbb{R}^{n})}\left[I(\varphi)+\Lambda(\varphi)\right]+\eta.
\end{align*}
Since $\eta$ is arbitrary, we complete the proof of the Laplace principle upper bound. \hspace{\fill}$\Box$

\section{Appendix}\label{s8.1}

\setcounter{equation}{0}
\setcounter{definition}{0}
\noindent\textbf{Proof of (\ref{z4.45}).}
We recall that
\begin{align*}
X_t^{\delta,\bar{h}^{\delta}}=&x+\int_{0}^{t}b(X_s^{\delta,\bar{h}^{\delta}},\mathscr{L}_{X_s^{\delta}},Y_s^{\delta,\bar{h}^{\delta}},\mathscr{L}_{Y_s^{\delta}})ds+\int_{0}^{t}\sigma(X_s^{\delta,\bar{h}^{\delta}},\mathscr{L}_{X_s^{\delta}},Y_s^{\delta,\bar{h}^{\delta}},\mathscr{L}_{Y_s^{\delta}})\bar{h}_s^{1}(Y_s^\delta)ds \nonumber\\
&+\sqrt{\delta}\int_{0}^{t}\sigma(X_s^{\delta,\bar{h}^{\delta}},\mathscr{L}_{X_s^{\delta}},Y_s^{\delta,\bar{h}^{\delta}},\mathscr{L}_{Y_s^{\delta}})dW_s^1 \nonumber\\
=:&x+\sum_{i=1}^3\mathcal{O}_i^\delta(t),
\end{align*}
and
\begin{align*}
\psi_{t}&=x+\int_0^t\bar{b}(\psi_{s},\mathscr{L}_{\bar{X}_{s}})ds+\int_0^t\int_\mathcal{Y}\sigma(\psi_{s},\mathscr{L}_{\bar{X}_{s}},y,\nu)\bar{h}_s^1(y)\nu(dy)ds,
\end{align*}
where $\bar{h}_s^1(y) := P_1 \bar{h}_s(y)$. It is clear  that
\begin{align}\label{z7.9}
X_t^{\delta,\bar{h}^{\delta}} - \psi_{t} = &\mathcal{O}_1^\delta(t) - \int_0^t\bar{b}(\psi_{s},\mathscr{L}_{\bar{X}_{s}})ds + \mathcal{O}_3^\delta(t)\nonumber\\
&+\mathcal{O}_2^\delta(t) - \int_0^t\int_\mathcal{Y}\sigma(\psi_{s},\mathscr{L}_{\bar{X}_{s}},y,\nu)\bar{h}_s^1(y)\nu(dy)ds.
\end{align}
On the one hand, it is easy to see that 
\begin{align}\label{z8.0}
\mathbb{E}\Big[\sup_{t\in[0,T]}|\mathcal{O}_3^\delta(t)|\Big]\leq C_T \delta^{\frac{1}{2}}.
\end{align}
 Moreover, using the same argument as in Lemma \ref{l4.2}, we can get
\begin{align}\label{z8.1}
\mathbb{E}\bigg[\sup_{t\in[0,T]}\left|\mathcal{O}_1^\delta(t)- \int_0^t\bar{b}(\psi_{s},\mathscr{L}_{\bar{X}_{s}})ds\right|\bigg] \leq C_{M,T}\mathbb{E}\int_0^T|X_s^{\delta,\bar{h}^\delta}-\psi_s|ds+\gamma_1(\delta),
\end{align}
where $\gamma_1(\delta)$ is a function satisfying  $\gamma_1(\delta) \to 0$, as $\delta \to 0$.
On the other hand,  we have
\begin{align}\label{z8.2}
&\mathbb{E}\bigg[\sup_{t\in[0,T]}\left|\mathcal{O}_2^\delta(t)-\int_0^t\int_\mathcal{Y}\sigma(\psi_{s},\mathscr{L}_{\bar{X}_{s}},y,\nu)\bar{h}_s^1(y)\nu(dy)ds\right|\bigg]\nonumber\\
\leq& \mathbb{E}\bigg[\sup_{t\in[0,T]}\left|\mathcal{O}_2^\delta(t)-\int_{0}^{t}\sigma(X_s^{\delta,\bar{h}^{\delta}},\mathscr{L}_{X_s^{\delta}},Y_s^{\delta,\bar{h}^{\delta}},\nu)\bar{h}_s^{1}(Y_s^\delta)ds\right|\bigg]\nonumber\\
&+\mathbb{E}\bigg[\sup_{t\in[0,T]}\left|\int_{0}^{t}\sigma(X_s^{\delta,\bar{h}^{\delta}},\mathscr{L}_{X_s^{\delta}},Y_s^{\delta,\bar{h}^{\delta}},\nu)\bar{h}_s^{1}(Y_s^\delta)ds-\int_0^t\int_\mathcal{Y}\sigma(\psi_{s},\mathscr{L}_{\bar{X}_{s}},y,\nu)\bar{h}_s^1(y)\nu(dy)ds\right|\bigg]\nonumber\\
=:&\mathcal{O}_{21}^\delta + \mathcal{O}_{22}^\delta.
\end{align}
It is evident that $\sigma(x,\mu,y,\nu)$ is Lipschitz continuous and bounded w.r.t.~$(x,\mu,y,\nu)$. Since $\bar{h}(y)$ is Lipschitz continuous
and bounded in $y \in \mathcal{Y}$, we conclude that  $\sigma(x,\mu,y,\nu)\bar{h}(z)$ is also Lipschitz continuous w.r.t.~$(x,\mu, y,\nu,z)$ and $\sigma(x,\mu,y,\nu)\bar{h}(y)$ is
locally Lipschitz continuous w.r.t.~$y$. Consequently, from the  same argument as in Lemma \ref{l4.1} we can deduce 
\begin{align}\label{z8.4}
\mathcal{O}_{21}^\delta \leq\gamma_2(\delta),
\end{align}
where $\gamma_2(\delta)$ is a function satisfying  $\gamma_2(\delta) \to 0$, as $\delta \to 0$.

In addition, from the same argument as in the proof of \cite[(3.38)]{cc}, we can get
\begin{align}\label{z8.3}
\mathcal{O}_{22}^\delta \leq C_{M,T}\mathbb{E}\bigg(\int_0^T|X_s^{\delta,\bar{h}^\delta}-\psi_s|ds\bigg) + \gamma_3(\delta),
\end{align}
where $\gamma_3(\delta)$ is a function satisfying  $\gamma_3(\delta) \to 0$, as $\delta \to 0$.

Collecting (\ref{z7.9})-(\ref{z8.3}) and using Gronwall's inequality, we deduce that (\ref{z4.45}) holds. \hspace{\fill}$\Box$   

\vspace{5mm}

\noindent\textbf{Acknowledgements} The authors would like to thank the referees for their very constructive suggestions.

\vspace{3mm}




\end{document}